\documentclass[a4paper,10pt]{article}
\usepackage{amssymb}
\usepackage{amsthm}
\usepackage{amsmath}
\usepackage{anysize}
\usepackage{url}
\usepackage{color}
\usepackage{subfigure,graphicx}
\usepackage{overpic}

\newcommand{\bitem}{\begin{itemize}}
\newcommand{\eitem}{\end{itemize}}
\newcommand{\beq}{\begin{equation}}
\newcommand{\eeq}{\end{equation}}
\newcommand{\beqn}{\begin{eqnarray*}}
\newcommand{\eeqn}{\end{eqnarray*}}

\newcommand{\cE}{{\cal E}}

\newcommand{\argmin}{\mbox{argmin}}

\newcommand{\norm}[1]{\|#1\|}

\def\RR{\mathbb{R}}

\def\epsilon{\varepsilon}

\newtheorem{theorem}{Theorem}[section]
\newtheorem{lemma}[theorem]{Lemma}

\newtheorem{definition}[theorem]{Definition}
\newtheorem{remark}[theorem]{Remark}

\newtheorem{prop}[theorem]{Proposition}
\newtheorem*{mtheorem}{Meta-Theorem}

\author{Philipp Grohs and Gitta Kutyniok}
\title{Parabolic Molecules\\
}
\begin{document}
\maketitle
\begin{abstract}
Anisotropic decompositions using representation systems based on parabolic scaling such as curve\-lets or shearlets
have recently attracted significantly increased attention due to the fact that they were shown to provide
optimally sparse approximations of functions exhibiting singularities on lower dimensional embedded manifolds.
The literature now contains various direct proofs of this fact and of related sparse approximation results.
However, it seems quite cumbersome to prove such a canon of results for each system separately, while many
of the systems exhibit certain similarities.

In this paper, with the introduction of the notion of {\em parabolic molecules}, we aim to provide
a comprehensive framework which includes customarily employed representation systems based on parabolic scaling
such as curvelets and shearlets. It is shown that
pairs of parabolic molecules have the fundamental property to be almost orthogonal in a particular sense.
This result is then applied to analyze parabolic molecules with respect to their ability to sparsely
approximate data governed by anisotropic features. For this, the concept of {\em sparsity equivalence}
is introduced which is shown to allow the identification of a large class of parabolic molecules
providing the same sparse approximation results as curvelets and shearlets. Finally, as another
application, smoothness spaces associated with parabolic molecules are introduced providing a general
theoretical approach which even leads to novel results for, for instance, compactly supported shearlets.
\end{abstract}
\begin{center}{\it Keywords: }{Curvelets, Nonlinear Approximation, Parabolic Scaling, Shearlets, Smoothness
Spaces, Sparsity Equivalence}
\end{center}
\begin{center}{\bf 2000 Mathematics Subject Classification.} Primary 41AXX, Secondary 41A25, 53B, 22E.
\end{center}


\section{Introduction}

Recently, a paradigm shift could be observed in applied mathematics, computer science, and electrical engineering.
The novel paradigm of sparse approximations now enables not only highly efficient encoding of functions
and signals, but also provides intriguing new methodologies, for instance, for recovery of missing data or
separation of morphologically distinct components. At about the same time, scientists began to question whether
wavelets are indeed perfectly suited for image processing tasks, the main reason being that images are governed
by edges while wavelets are isotropic objects. This mismatch becomes also evident when recalling that Besov
spaces can be characterized by the decay of wavelet coefficient sequences however Besov models are clearly deficient
to adequate capturing of edges.

\subsection{Geometric Multiscale Analysis}

These two fundamental observations have led to the research area of geometric multiscale analysis whose
main goal is to develop representation systems, preferably containing different scales, which are sensitive
to anisotropic features in functions/signals and provide sparse approximations of those. Such representation
systems are typically based on {\em parabolic scaling}, and we exemplarily mention (first and second generation)
curvelets  \cite{CD04}, contourlets \cite{DV05}, and shearlets \cite{KL12}. Browsing through
the literature, it becomes evident that some properties such as sparse approximation of so-called cartoon
images are quite similar for some systems such as curvelets and shearlets, whereas other systems such as
contourlets show a somehow different behavior. Delving more into the literature we observe that for those systems
exhibiting similar sparsity behavior many results were proven with quite resembling proofs. One might ask:
Is this cumbersome close repetition of proofs really necessary? We believe that the answer is {\em no} and
that a general framework for representation systems based on parabolic scaling does not only solve this problem
but even more provide a fundamental understanding of such systems and allow for a categorization of these.

\subsection{Parabolic molecules}

The main goal of this paper is to proclaim the framework of parabolic molecules as a general concept
encompassing in particular curvelets and both band-limited and compactly supported shearlets. The
idea of {\em molecules} in geometric multiscale analysis dates back to the seminal work by Cand\`{e}s and
Demanet \cite{CD02}, in which they studied the curvelet representation of wave propagators by using
what they called {\em curvelet molecules}. Later, Guo and Labate adopted this idea and introduced
{\em shearlet molecules} in \cite{GL08}.

Both such generalization approaches however suffer from the fact that they are solely designed to
weaken the conditions of the particular respective systems, namely curvelets and shearlets. In contrast
to this, our philosophy is to introduce molecules, which encompass a wide class of directional
representation systems by using parabolic scaling as a unifying concept. This is justified by
the fact that all known systems providing optimally sparse approximation of cartoon images follow
a parabolic scaling law; and it is strongly believed that this is necessary. In fact, our
framework is general enough to, for instance, include all known curvelet-type as well as shearlet-type
constructions to date.

Our main result (Theorem \ref{thm:almostorth}) will show that the Gramian matrix between any two systems
of parabolic molecules satisfies a strong off-diagonal decay property and is in that sense very close to
a diagonal matrix. This will become key to transfer the celebrated properties of curvelet systems to
other systems based on parabolic scaling; a fact which we can summarize in the following meta-result:

\begin{mtheorem}
    All frame systems based on parabolic scaling (specifically curvelets and shearlets) posses the exact same approximation
    properties, whenever the generating functions are sufficiently smooth, as well as localized in space and frequency.
\end{mtheorem}

This meta-theorem has been verified on a case-by-case basis for a number of different systems. In this paper,
for the first time, a rigorous framework is provided which applies to, for instance, all known curvelet
or shearlet constructions at once. This will be exemplarily demonstrated by the results on sparse approximation
(Theorem \ref{theo:mainsparsity}) and anisotropic smoothness spaces (Theorem \ref{thm:normequiv}) which are
indeed universally applicable to all parabolic molecules.

\subsection{Sparsity Equivalence}

Focussing on the property of sparse approximation of images governed by anisotropic features, it might
be even more beneficial to derive a categorization of parabolic molecules according to their approximation
behavior. We accommodate this request by introducing the notion of {\em sparsity equivalence} in Subsection
\ref{subsec:sparse}, which leads to equivalence classes and further to the sought classification.
As a byproduct, our framework yields a simple derivation of the results in \cite{GL07,Kutyniok2010}
from \cite{CD04}. In fact, our results provide a systematic way to analyze the sparse approximation
of cartoon images of systems by elements of the class of parabolic molecules categorized by equivalence
classes of sparsity equivalence.

\subsection{Contribution and Expected Impact}

Summarizing, the significance of the notion of parabolic molecules as a higher level viewpoint lies in the fact that it
not only provides a general framework which contains various directional representation systems as special
cases and enables a quantitative comparison of such, but it moreover allows the transfer of results concerning
properties of such systems without repeating quite similar proofs. A few examples, for which this conceptually
new approach is fruitful, will be presented in Section \ref{sec:appl} including optimally sparse
approximations of cartoon-like images.

We therefore anticipate this novel framework to have the following impacts:
\bitem
\item A thorough understanding of the ingredients of representation systems based on parabolic scaling which
are crucial for an observed behavior such as sparse approximations of cartoon images, thereby also
categorizing different (sparsity) behaviors.
\item A framework within which results can be directly transferred from one system to others without
repetition of similar proofs. This will allow to establish a desired result for a system based on parabolic
scaling by choosing, for instance, a shearlet or curvelet system best suited for the proof, and transfer the
result subsequently to any other systems by utilizing the results in this paper.
\item An approach to design new representation systems based on parabolic scaling depending on several
parameters whose impact on, for instance, sparse approximation behavior is then known in advance.
\eitem

\subsection{Extensions}

The framework introduced in this paper and the derived results are amenable to generalizations and extensions.
We briefly discuss a few examples.

\bitem
\item {\em Other Systems}. This general framework supports the introduction of novel systems based on
parabolic scaling fulfilling a list of desiderata designed according to a particular application. Such systems
can now be objectively compared with other systems according to, for instance, sparse approximation properties.
\item {\em Systems with Continuous Parameters}. One might also ask whether a similar general framework
for systems based on parabolic scaling with continuous parameters can be introduced. In light of Subsection \ref{subsec:sparse},
this however requires a different sparsity model; one conceivable path would be to compare their ability to
resolve wavefront sets.
\item {\em Further Properties}. In this paper, we studied the impact of our general framework on the
problems of sparse approximation and anisotropic function spaces. This strategy can certainly be also
used for other applications such as efficient decomposition of the Radon transform, which has been
studied both for shearlets \cite{CBGL10} and curvelets \cite{CD00b}, as well as the analysis
of geometric separation as studied in \cite{Donoho2010c}.
\item {\em Weighted Norms}. When aiming at transferring results such as sparse decompositions of curvilinear
integrals \cite{CD00a} or sparse decompositions of the Radon transform \cite{CD00b}, sometimes weighted
$\ell_p$ norms might need to be analyzed.
\item {\em Higher Dimensions}. We have formulated our results in the bivariate setting. However, an
extension to arbitrary dimensions is possible using essentially the same arguments. This is especially
relevant since by now several different curvelet and shearlet constructions exist for three-dimensional
data \cite{Ying2005,Kutyniok2012,Guo2010}.
\eitem

\subsection{Outline}

This paper is organized as follows. In Section \ref{sec:paramol}, the notion of parabolic molecules
is introduced. It is then shown in Section \ref{sec:exam} that curvelets and both band-limited and
compactly supported shearlets are special cases of this framework. Almost orthogonality of pairs
of parabolic molecules is proven in Section \ref{sec:proof}. Finally, in Section \ref{sec:appl},
this result is utilized for two applications. First, in Subsection \ref{subsec:sparse}, using
the novel concept of sparsity equivalence a large class of parabolic molecules providing the same
sparse approximation results as curvelets and shearlets is identified. Second, in Subsection
\ref{sec:functionspaces}, smoothness spaces associated with parabolic molecules are studied.

\subsection{Notation}
We comment on the notation which we shall use in the present work. Denote by
$L_p(\mathbb{R}^d)$ the usual Lebesgue spaces with associated norm $\|\cdot \|_p$.
For a discrete set $\Lambda$ equipped with the counting measure we denote the corresponding
Lebesgue space by $\ell_p(\Lambda)$ or $\ell_p$ if $\Lambda$ is known from the context. The associated
norm will again be denoted $\|\cdot \|_p$.
We use the symbol $\langle \cdot , \cdot \rangle$ indiscriminately for the inner product on the
Hilbert space $L_2(\mathbb{R}^d)$ as well as for the Euclidean inner product on $\mathbb{R}^d$.
The Euclidean norm $\langle x , x\rangle^{1/2}$ of a vector $x\in \mathbb{R}^d$ will be denoted
by $|x|$.
For a function $f\in L_1(\mathbb{R}^d)$ we can define
the Fourier transform $\hat f(\omega):=\int_{\mathbb{R}^d}f(x)\exp(-2\pi i \langle x,\omega\rangle)\mathrm{d}x$.
By density this definition can be extended to tempered distributions $f$.
We shall also use the notation $\mathbb{T}$ to denote the one-dimensional torus which can be identified
with the half-open interval $[0,2\pi)$. Sometimes we will use the notations $(x)_+:=\max(x,0)$
$\lfloor x \rfloor :=\max\{l\in \mathbb{Z}: l\le x\}$, and
$\langle x \rangle := (1 + x^2 ) ^{1/2}$. Finally, we use the symbol $A\lesssim B$ to indicate
that $A\le CB$ with a uniform constant $C$.

\section{Parabolic Molecules}
\label{sec:paramol}
All anisotropic transforms based on parabolic scaling which have appeared in the literature
are indexed by a scale parameter describing the amount of anisotropic scaling,
an angular parameter describing the orientation and a spacial parameter describing the location of an element.
Nevertheless, these specific constructions are based on different principles: For curvelets the
scaling is done by a dilation with respect to polar coordinates and the orientation is enforced by rotations.
Shearlets on the other hand are based on affine scaling of a single generator and the directionality
is generated by the action of shear matrices.

It is the purpose of this section to introduce the concept of parabolic molecules which
distills the essential properties out of these constructions in terms of time-frequency localization
properties. As it will turn out, all previous constructions of curvelets and shearlets are
instances of this concept, a fact that enables us to operate in a much more general setup than
in previous work.
%
\subsection{Definition of Parabolic Molecules}
%
Let us now describe our setup. We start by defining our parameter space
$$
    \mathbb{P}:= \mathbb{R}_+\times\mathbb{T}\times \mathbb{R}^2,
$$
where a point $p = (s,\theta,x)\in \mathbb{P}$ describes
a scale $2^s$, an orientation $\theta$, and a location $x$.

Parabolic molecules are defined as systems of functions $(m_{\lambda})_{\lambda\in \Lambda}$
with each $m_{\lambda}\in L_2(\mathbb{R}^2)$ satisfying some additional properties.
In particular, each function $m_{\lambda}$ will be associated with a unique point in $\mathbb{P}$, as we shall
make precise below.
Since we are dealing with discrete systems (frames) we would like to
operate with discrete sampling sets contained in $\mathbb{P}$.
We call such sampling sets \emph{parametrizations} as defined below.
\begin{definition}A \emph{parametrization} consists of a
    pair $(\Lambda,\Phi_\Lambda)$
    where $\Lambda$ is a discrete index set and $\Phi_\Lambda$ is a mapping
    $$
        \Phi_\Lambda:\left\{\begin{array}{ccc}\Lambda &\to & \mathbb{P},\\
        \lambda \in\Lambda & \mapsto & \left(s_\lambda , \theta_\lambda , x_\lambda\right).
        \end{array}\right.
    $$
    which associate with each $\lambda\in \Lambda$ a \emph{scale} $s_\lambda$,
    a \emph{direction} $\theta_\lambda$ and a \emph{location} $x_\lambda\in \mathbb{R}^2$.

    There exists a \emph{canonical parametrization}
    $$
        \Lambda^0:=\left\{ (j,l,k)\in \mathbb{Z}^4\ :\
        j\geq 0,\ l  = - 2^{\lfloor \frac{j}{2}\rfloor-1}, \cdots
        , 2^{\lfloor \frac{j}{2}\rfloor-1}\right\},
    $$
    where for $\lambda = (j,l,k)$ we define $\Phi^0(\lambda):=(s_\lambda,\theta_\lambda,x_\lambda)$
    with
    $s_\lambda  := j$, $\theta_\lambda:=l 2^{-\lfloor j/2\rfloor }\pi$, and
    $x_\lambda := R_{-\theta_\lambda}D_{2^{-s_\lambda}}k$.
\end{definition}

We are now ready to define parabolic molecules. Our definition
essentially says that molecules have frequency support in parabolic
wedges associated to a certain orientation and spacial support in
rectangles with parabolic aspect ratio.

For this, we will use the following notion. We let
$R_\theta = \left(\begin{array}{cc}\cos(\theta)&-\sin(\theta)\\\sin(\theta)&\cos(\theta)\end{array}\right)$
denote the rotation matrix of angle $\theta$,
and $D_a:=\mbox{diag} ( a, \sqrt{a})$ be the anisotropic dilation matrix associated with $a>0$.
\begin{definition}Let $\Lambda$ be a parametrization.
    A family $(m_\lambda)_{\lambda \in \Lambda}$ is called a \emph{family of parabolic molecules} of
    order $(R,M,N_1,N_2)$ if
    it can be written as
    $$
        m_\lambda (x) =
        2^{3s_\lambda/4}
        a^{(\lambda)}
        \left(D_{2^{s_\lambda}}R_{\theta_\lambda}\left(x - x_\lambda\right)\right)
    $$
    such that
    \begin{equation}\label{eq:molcond}
        \left| \partial^\beta \hat a^{(\lambda)}(\xi)\right|
        \lesssim \min\left(1,2^{-s_\lambda} + |\xi_1| + 2^{-s_\lambda/2}|\xi_2|\right)^M
        \left\langle |\xi|\right\rangle^{-N_1} \langle \xi_2 \rangle^{-N_2}
    \end{equation}
    for all $|\beta|\le R$.
    The implicit constants are uniform over $\lambda\in \Lambda$.
\end{definition}
\begin{remark}
    For convenience our definition only poses conditions on the Fourier transform of $m_\lambda$.
    The number $R$ describes the spatial localization, $M$ the number of directional (almost)
    vanishing moments and $N_1,N_2$ describe the smoothness of an element $m_\lambda$.
    We also refer to Figure \ref{fig:molfreqsupport} for an illustration of the approximate frequency
    support of a parabolic molecule.
\end{remark}
\begin{figure}[ht]
\centering
\subfigure{
\includegraphics[width = .45\textwidth]{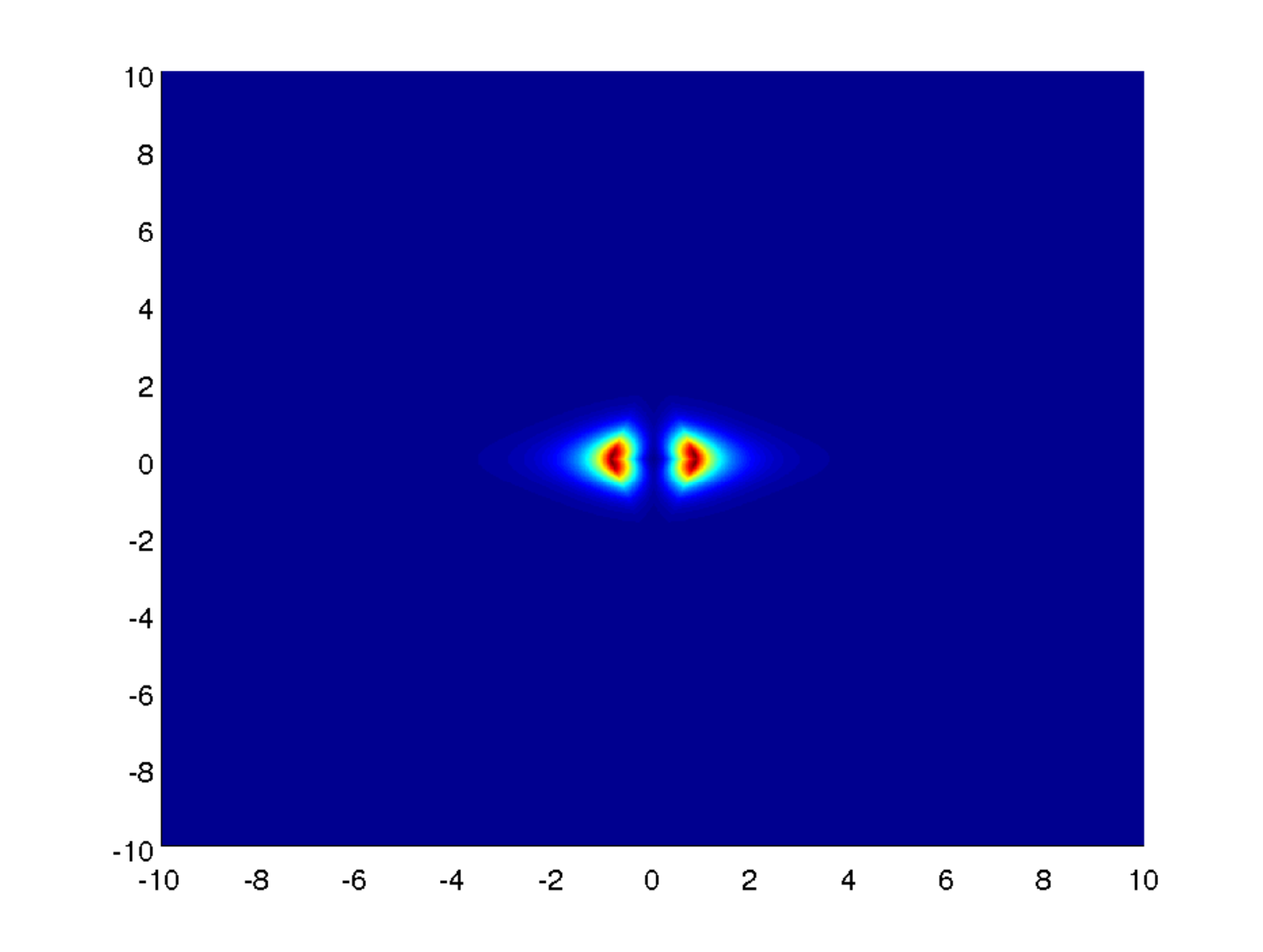}
}
\subfigure{
\includegraphics[width = .45\textwidth]{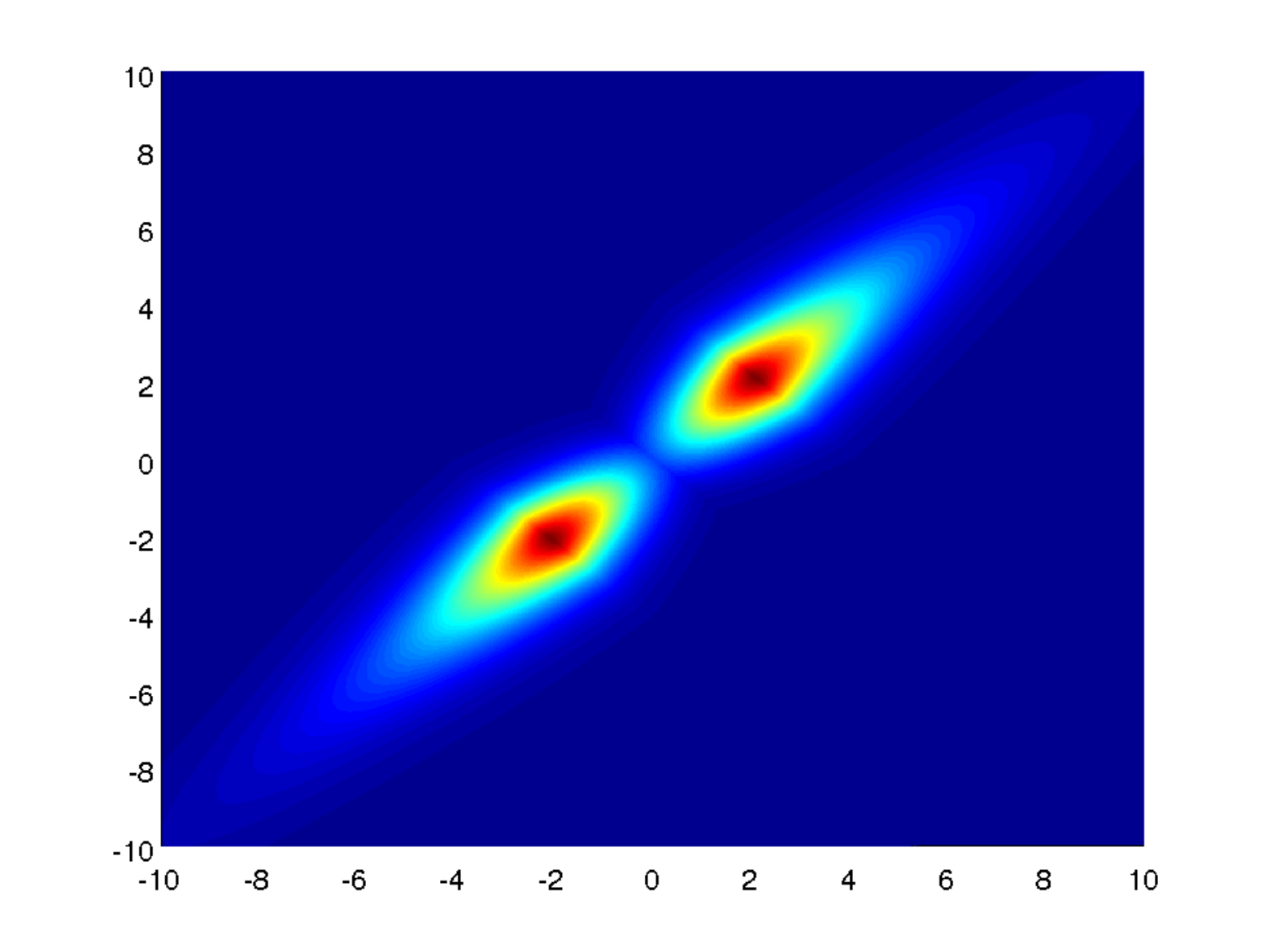}
}
\caption{Left: The weight function $\min
\left(1,2^{-s_\lambda} + |\xi_1| + 2^{-s_\lambda/2}|\xi_2|\right)^M
\left\langle |\xi|\right\rangle^{-N_1} \langle \xi_2 \rangle^{-N_2}$
for $s_\lambda = 3$, $M = 3$, $N_1=N_2=2$. Right: Approximate Frequency
support of a corresponding molecule $\hat{m}_{\lambda}$ with $\theta_\lambda = \pi/4$.}
\label{fig:molfreqsupport}
\end{figure}
We pause to record the following simple estimates:
In polar coordinates we have the representation
\begin{equation*}
    \hat m_{\lambda}(r,\varphi) = 2^{-3s_\lambda / 4}\hat a^{(\lambda)}\left(
    2^{-s_\lambda} r \cos(\varphi + \theta_\lambda),2^{-s_\lambda/2}r
    \sin(\varphi + \theta_\lambda)\right)\exp\left(2\pi i \langle x_\lambda, \xi\rangle\right).
\end{equation*}
Equation (\ref{eq:molcond}) directly implies that in polar coordinates we have the estimate
\begin{equation}\label{eq:moldecaypolar}
    \left|\hat m_{\lambda}(\xi)\right|
    \lesssim  2^{-2s_\lambda/4}\min\left(1, 2^{-s_\lambda}( 1 + r )\right)^M
    \left\langle 2^{-s_\lambda}r\right \rangle^{-N_1}
    \langle 2^{-s_\lambda/2}r\sin(\varphi + \theta_\lambda)\rangle^{-N_2}.
\end{equation}
%
%
\subsection{Metric Properties of Parametrizations}
%
In order to proceed we need to introduce some additional (metric) properties of index sets and
parametrizations.
The parameter space $\mathbb{P}$ can be equipped with a natural notion of (pseudo) distance,
see \cite{Smith1998a}, which can be extended to a distance between indices
by a pullback via a parametrization.
\begin{definition} \label{defi:indexdistance}
    Following \cite{CD02,Smith1998a}, we define for two indices $\lambda,\mu$
    the \emph{index distance}
    $$
        \omega \left(\lambda,\mu\right):=2^{\left|s_\lambda - s_{\mu}\right|}
        \left(
        1 + 2^{s_{\lambda_0}}d\left(\lambda,\mu\right)\right),
    $$
    and
    $$
        d\left(\lambda,\mu \right):=|\theta_\lambda - \theta_{\mu}|^2 + |x_\lambda
         - x_{\mu}|^2 +
        |\langle e_\lambda , x_\lambda - x_{\mu}\rangle|.
    $$
    where $\lambda_0 = \mbox{argmin}(s_\lambda,s_{\mu})$
    and $e_\lambda = \left(\cos(\theta_\lambda),\sin(\theta_\lambda)\right)^\top$.
\end{definition}
\begin{remark}
    The notation $\omega(\lambda,\mu)$ is actually a slight abuse of notation since
    $\omega$ is acting on $\mathbb{P}$. Therefore it should read
    $$
        \omega\left(\Phi_\Lambda(\lambda), \Phi_M(\mu)\right)
    $$
    for indices $\lambda\in \Lambda$, $\mu\in M$ with
    associated parametrizations $\Phi_\Lambda$, $\Phi_M$.
    In order not to overload the notation we stick with the shorter but slightly less accurate definition.
\end{remark}
\begin{remark}
    We wish to mention that, in fact, real-valued curvelets or shearlets are not associated
    with an angle but with a ray, i.e., $\theta$ and $\theta + \pi$ need to be identified.
    This is not reflected in the above definition, which is a slight inaccuracy. The 'correct'
    definition should assume that $|\theta_\lambda|\le \frac{\pi}{2} \in \mathbb{P}^1$, the
    projective line. Therefore, it should read
    $$
        d\left(\lambda,\mu \right):=|\{\theta_\lambda - \theta_{\mu}\}|^2 + |x_\lambda
         - x_{\mu}|^2 +
        |\langle \theta_\lambda , x_\lambda - x_{\mu}\rangle|
    $$
    with
    $\{\varphi \}$ the projection of $\varphi$ onto $\mathbb{P}^1\cong (-\pi/2 , \pi/2]$.

    However, for our results it will make no difference which definition is used. Thus we
    decided to employ Definition \ref{defi:indexdistance}, which avoids additional
    technicalities.
\end{remark}
We need to impose further conditions on an index set $\Lambda$ in order to arrive at meaningful results.
The following definition formalizes a crucial property, which is later on required to be satisfied by an
index set in our results.
\begin{definition}\label{def:kadmiss}
    An index set $\Lambda$ with associated mapping $\Phi_\Lambda$ is called
    \emph{$k$-admissible}
    if
    \begin{equation*}
        \sup_{\lambda\in \Lambda}\sum_{\mu\in\Lambda^0}
        \omega\left(\lambda,\mu\right)^{-k} <\infty,
    \end{equation*}
    and
    \begin{equation*}
        \sup_{\lambda\in \Lambda^0}\sum_{\mu\in\Lambda}
        \omega\left(\lambda,\mu\right)^{-k} <\infty.
    \end{equation*}
\end{definition}
\begin{lemma}\label{lem:canonadmin}
 The canonical index set $\Lambda^0$ is $k$-admissible for all $k>2$.
\end{lemma}
\begin{proof}
    We aim to prove that
    \begin{equation}\label{eq:prove}
        \sup_{\mu\in \Lambda^0}\sum_{\lambda\in\Lambda^0}
        \omega\left(\mu,\lambda\right)^{-k} <\infty.
    \end{equation}
    Writing $s_{\mu} = j'$ in the definition of $\omega\left(\mu,\lambda\right)$, we need to consider
    \begin{equation}\label{eq:prove2}
        \sum_{j\in\mathbb{Z}_+}\sum_{\lambda \in \Lambda^0, s_\lambda  = j}
        2^{-k\left|j - j' \right|}\left(1 +
        2^{\min(j,j')}d(\mu,\lambda)\right)^{-k}.
    \end{equation}
    According to \cite[Equation (A.2)]{CD02}, we have
    \begin{equation}\label{eq:cd2}
        \sum_{\lambda \in \Lambda^0, s_\lambda  = j}(1 + 2^{q}d(\mu,\lambda))^{-2}\lesssim
        2^{2(j-q)_+}
    \end{equation}
    for any $q$. Hence, for each $k>2$, \eqref{eq:prove2} can be estimated by
    $$
        \sum_{j\geq 0}2^{-k|j-j'|}2^{2|j-j'|} <\infty,
    $$
    which proves \eqref{eq:prove}.
\end{proof}
%
%
\subsection{Main Result}
%
The main result of this paper essentially states that any two systems of parabolic
molecules behave in the same way as far as approximation properties are concerned.
Specifically, we show the following theorem, whose proof is quite technical, wherefore
we postpone it to Subsection \ref{sec:almostorth}.
\begin{theorem}\label{thm:almostorth}
     Let $(m_\lambda)_{\lambda\in \Lambda}$,
     $(p_{\mu})_{\mu\in M}$ be two systems of
     parabolic molecules of order $(R,M,N_1,N_2)$
     with
     \begin{equation}\label{eq:conds}
        R\geq 2N,\quad  M > 4N-\frac{5}{4} ,\quad N_1 \geq 2N+\frac{3}{4} , \quad
        N_2\geq 2N.
     \end{equation}
     Then
     $$
        \left|\left\langle m_\lambda ,p_{\mu}\right\rangle\right|\lesssim
        \omega\left((s_\lambda,\theta_\lambda,x_\lambda),(s_\mu,\theta_\mu,x_\mu )\right)^{-N}.
     $$
\end{theorem}
This result shows that the Gramian matrix between any two systems of parabolic molecules
satisfies a strong off-diagonal decay property and is in that sense very close to
a diagonal matrix. As we shall see in Section \ref{sec:appl}, this result has a number of
immediate applications, most notably for the approximation properties of arbitrary frames which
are systems of parabolic molecules (they turn out to be equivalent!).

We find it particularly striking that our framework is general enough to include both curvelet-type,
as well as shearlet-type constructions (see Section \ref{sec:exam}). Therefore, as a consequence
of Theorem \ref{thm:almostorth}, all these systems satisfy the same celebrated properties of the
curvelet construction given in \cite{CD04}. To demonstrate the importance of our result, Section
\ref{sec:appl} discusses selected applications of Theorem \ref{thm:almostorth} such as sparsity
equivalence and equivalence of associated smoothness spaces.

%
\section{Examples of Parabolic Molecules}
\label{sec:exam}
%
Having defined parabolic molecules in Section \ref{sec:paramol} above, it is important to
examine the versatility of this concept. This is done in the present section.
The main findings are that essentially all known constructions in the literature
can be cast in our framework and are thus amenable to the techniques and results
developed in this paper.

We divide the section into two subsections. In Subsection \ref{sec:examcurve}
we study so-called curvelet-like constructions. These include curvelets as
defined in \cite{Candes2005a} but also other constructions, such as in \cite{Borup2007,Smith1998a}.
We show that all these function systems are parabolic molecules. In fact, this result should not
come to much as a surprise: In \cite{CD02} a similar concept of curvelet molecules is introduced which
includes all the above-mentioned constructions. We also show that curvelet molecules are parabolic molecules.

The real strength of our definition of parabolic molecules is that it includes not only curvelet-type constructions.
In fact, we consider it one of the main findings of this paper that also shearlet-type systems can be thought of
as instances of parabolic molecules, associated to a specific shearlet parametrization $\Phi^\sigma$!
We show this result, as well as the admissibility of $\Phi^\sigma$, below in Subsection \ref{sec:examshear}.
After that, to provide some concrete examples, we study several specific constructions. In particular, we show
that compactly supported shearlet constructions (see e.g. \cite{Kittipoom2010}) are parabolic molecules.
%
\subsection{Curvelet-like constructions}\label{sec:examcurve}
%
%
\subsubsection{Second Generation Curvelets}\label{sec:curvelets}
%
It is easily verified that curvelet molecules as defined in \cite{CD02} are instances of
parabolic molecules associated with the canonical parametrization. In particular,
second generation curvelets \cite{Candes2005a} are parabolic molecules of arbitrary order.
We start by describing the construction.
Pick two window functions $W(r)$, $V(t)$ which are both real, nonnegative, $C^\infty$
and supported in $\left( \frac{1}{2} , 2\right)$ and in $\left(-1 , 1 \right)$ respectively.
We further assume that these windows satisfy
\begin{equation*}
    \sum_{j\in\mathbb{Z}}W\left(2^j r \right)^2 = 1\quad \mbox{for all }r\in\mathbb{R}_+
    \quad \mbox{and}\quad \sum_{l\in\mathbb{Z}}V\left(t-l\right)^2=1\quad
    \mbox{for all }t\in\left(-\frac{1}{2},\frac{1}{2}\right).
\end{equation*}
Now define in polar coordinates
\begin{equation*}
    \hat\gamma_{(j,0,0)}(r,\omega):=2^{-3j/4}W\left(2^{-j}r\right)V\left(2^{\lfloor j/2 \rfloor}\omega \right)
    \mbox{ and }
    \gamma_{(j,l,k)}(\cdot):=\gamma_{(j,0,0)}\left(R_{\theta_{(j,l,k)}}
    \left(\cdot - x_{(j,l,k)}\right)\right),
\end{equation*}
where
$(j,l,k)\in \Lambda^0$.
With appropriate modifications for the low-frequency case $j=0$
it is possible to show that the system
$$
    \Gamma^0:=\left\{\gamma_\lambda: \lambda\in \Lambda^0 \right\}
$$
constitutes a Parseval frame for $L_2(\mathbb{R}^2)$.
In order to make the frame elements real-valued, it is possible
to identify elements oriented in antipodal directions.
This frame is customarily referred to as
the tight frame of \emph{second generation curvelets}.
We now show that this frame forms a system of parabolic molecules of arbitrary order.
\begin{prop}
    The second generation curvelet frame constitutes a system of
    parabolic molecules of arbitrary order associated with the canonical
    parametrization.
\end{prop}
\begin{proof}
    Due to rotation invariance we may restrict ourselves to the case $\theta_\lambda = 0$.
    Therefore, denoting $\gamma_j:=\gamma_{(j,0,0)}$, we need to show that the function
    $$
          a^{(\lambda)}(\cdot):=2^{-3s_\lambda/4}\gamma_j\left(D_{2^{-s_\lambda}}\cdot \right)
    $$
    satisfies (\ref{eq:molcond}) for $(R,M,N_1,N_2)$ arbitrary.
    First note that
    \begin{equation*}
        \hat a^{(\lambda)}(\cdot)=2^{3s_\lambda/4}\hat \gamma_j\left(D_{2^{s_\lambda}}\cdot \right).
    \end{equation*}
    The function $\hat a^{(\lambda)}$, together with all its derivatives has
    compact support in a rectangle away from the $\xi_1$-axis.
    Therefore, it only remains to show that, on its support,
    the function $\hat a^{(\lambda)}$ has bounded derivatives, with
    a bound independent of $j$.
    But this follows from elementary arguments, using
    $r = \sqrt{\xi_1^2 + \xi_2^2}$, $\omega = \arctan\left(\xi_2/\xi_1\right)$.
    Then
    $$
        \hat a^{(\lambda)}(\xi)
        = \hat \gamma_{(j,0,0)}\left(D_{2^{j}}\xi\right)
        = W\left(\alpha_j(\xi)\right)V\left(\beta_j(\xi)\right),
    $$
    where
    $$
        \alpha_j(\xi) := 2^{-j}\sqrt{2^{2j}\xi_1^2+2^j\xi_2^2},\mbox{ and }
        \beta_j(\xi):= 2^{j/2}\arctan\left(\frac{\xi_2}{2^{j/2}\xi_1}\right).
    $$
    Now it is a simple calculus exercise to show that all derivatives of $\alpha_j$
    and $\beta_j$ are bounded on the support of $\hat a^{(\lambda)}$ and
    uniformly in $j$. This proves the result.
\end{proof}
%
%
\subsubsection{Hart Smith's Parabolic Frame}\label{sec:smithframe}
%
Historically, the first instance of a decomposition into parabolic molecules can be found in
Hart Smith's work on Fourier Integral Operators and Wave Equations \cite{Smith1998a}.
This frame, as well as its dual, again forms a system of parabolic molecules of
arbitrary order associated with the canonical parametrization.
We refer to \cite{Smith1998a,Andersson2008} for the details of the
construction which is essentially identical
to the curvelet construction, with primal and dual frame being allowed to differ.
The same discussion as above for curvelets shows
that this system consists of parabolic molecules.
%
\subsubsection{Borup and Nielsen's Construction}\label{sec:borupframe}
%
Another very similar construction has been given in \cite{Borup2007}. In this paper, the focus has been on
the study of associated function spaces. Again, it is straightforward to prove that this system constitutes
a system of parabolic molecules of arbitrary order associated with the canonical parametrization.
As a corollary to our results, it will actually turn out that
the spaces defined in \cite{Borup2007} coincide with the approximation spaces corresponding to
curvelets, shearlets, and Smith's transform.
%
\subsubsection{Curvelet Molecules}
%
In \cite{CD02} the authors introduced the notion of \emph{curvelet molecules} which are a useful concept
in proving sparsity properties of wave propagators. For the sake of completion, we include the exact
definition.
\begin{definition}
    Let $\Lambda^0$ be the canonical parametrization.
    A family $(m_\lambda)_{\lambda \in \Lambda^0}$ is
    called a \emph{family of curvelet molecules} of
    regularity $R$ if
    it can be written as
    $$
        m_\lambda (x) =
        2^{3s_\lambda/4}
        a^{(\lambda)}
        \left(D_{2^{s_\lambda}}R_{\theta_\lambda}\left(x - x_\lambda\right)\right)
    $$
    such that for all $|\beta| \le R$ and each $N = 0,1,2,\dots$
    \begin{equation}\label{eq:curvemoldecay}
        |\partial^\beta a^{(\lambda)}(x)|\lesssim \langle x\rangle^{-N}
    \end{equation}
    and for $M = 0,1,\dots $
    \begin{equation}\label{eq:curvemolmoments}
        |\hat{a}^{(\lambda)}(\xi)|\lesssim
        \min\left(1,2^{-s_\lambda} + |\xi_1| + 2^{-s_\lambda/2}|\xi_2|\right)^M.
    \end{equation}
\end{definition}
This definition is similar to our definition of parabolic molecules, however with two crucial differences:
First, (\ref{eq:molcond}) allows for arbitrary rotation angles and is therefore more general. Curvelet
molecules on the other hand are only defined for the canonical parametrization $\Lambda^0$ (which, in contrast
to our definition, is not sufficiently general to also cover shearlet-type systems). Second, the decay conditions
analogous to our condition (\ref{eq:molcond}) are more restrictive in the sense that it requires infinitely
many nearly vanishing moments. In fact, the following result holds:
\begin{prop}\label{lem:curvemol}
    A system of curvelet molecules of regularity $R$ constitutes a system of
    parabolic molecules of order $(\infty, \infty , R/2, R/2)$.
\end{prop}
\begin{proof}
    The definition of curvelet molecules as above implies that
    the estimate (\ref{eq:curvemolmoments}) also holds for
    all derivatives of $\hat{a}^{(\lambda)}$, see \cite{CD02}.
    Furthermore, by (\ref{eq:curvemoldecay}), all derivatives of $\hat{a}^{(\lambda)}$
    can be estimated in modulus by $\langle |\xi| \rangle^{-R}$, which in turn can be estimated
    by $\langle |\xi| \rangle^{-R/2}\langle \xi_2\rangle^{-R/2}$. This yields the desired estimate.
\end{proof}
%
\subsection{Shearlets}\label{sec:examshear}
%
Shearlets were introduced in 2006 as the first directional representation system which not only
satisfies the same celebrated properties of curvelets, but also provides a unified treatment
of the continuum and digital setting. This key property is achieved through utilization of a
shearing matrix as a means to parameterize orientation, which is highly adapted to the digital
grid in contrast to rotation. For more information on shearlets, we refer to the book \cite{KL12}.

It is perhaps not surprising that curvelets and their relatives described above
fall into the framework of parabolic molecules.
Here we show the crucial fact that shearlets can be seen as a special case of parabolic
molecules as well.
Consider the discrete index set
\begin{equation}\label{eq:shearindex}
    \Lambda^\sigma := \left\{ (\varepsilon, j,l,k)\in \mathbb{Z}_2\times
    \mathbb{Z}^4:\varepsilon \in \{0,1\},\
    j\geq 0,\ l  = -2^{\lfloor \frac{j}{2}\rfloor}, \cdots
    , 2^{\lfloor \frac{j}{2}\rfloor}\right\},
\end{equation}
and the shearlet system
$$
    \Sigma:=\left\{\sigma_\lambda: \lambda\in \Lambda^\sigma\right\},
$$
with
$$
    \sigma_{(\varepsilon , 0 , 0 , k)}(\cdot) = \varphi(\cdot - k),\quad
    \sigma_{(\varepsilon , j, l , k)}(\cdot) = 2^{3j/4}\psi^\varepsilon_{j,l,k}
     \left(D^\varepsilon_{2^j}
    S_{l,j}^\varepsilon \cdot - k\right),\quad j\geq 1,
$$
where $D^0_{a}=D_a$, $D^1_a := \mbox{diag}(\sqrt{a},a)$,
$S_{l,j} := \left(\begin{array}{cc}1 & l 2^{-\lfloor j/2\rfloor}\\0&1\end{array}\right)$
and $S_{l,j}^1 = \left(S_{l,j}^0\right)^\top$. Then we define shearlet molecules  of order
$(R,M,N_1,N_2)$, which is a generalization of shearlets adapted to parabolic molecules, in
particular including the classical shearlet molecules introduced in \cite{GL08}, see
Subsection \ref{subsec:shearletmole}.
\begin{definition}
    We call $\Sigma$ a system of \emph{shearlet molecules}
    of order $(R,M,N_1,N_2)$
    if
    the functions $\varphi,\ \psi^0_{j,l,k},\ \psi^1_{j,l,k}$ satisfy
    \begin{equation}\label{eq:shearcond}
        |\partial^\beta \hat \psi^\varepsilon_{j,l,k}(\xi_1,\xi_2)|\lesssim \min
        \left(1,|\xi_{1+\varepsilon}|\right)^M
        \langle |\xi |\rangle^{-N_1}\langle \xi_{2-\varepsilon}\rangle^{-N_2}
    \end{equation}
    for every $\beta\in \mathbb{N}^2$ with $|\beta|\le R$.
\end{definition}
\begin{remark}
    In our proofs it is nowhere required that the directional parameter $l$
    runs between $-2^{\lfloor \frac{j}{2}\rfloor}$ and $-2^{\lfloor \frac{j}{2}\rfloor}$
    Indeed, $l$ running in any discrete interval
    $-C2^{\lfloor \frac{j}{2}\rfloor},\dots , C2^{\lfloor \frac{j}{2}\rfloor}$
    would yield the exact same results, as a careful inspection of our arguments
    shows. Likewise, in certain shearlet constructions, the translational
    sampling runs not through $k\in \mathbb{Z}^2$ but through $\tau \mathbb{Z}^2$
    with $\tau>0$ a sampling constant. Our results are also valid for this case with
    the similar proofs. The same remark applies to all curvelet-type constructions.
\end{remark}
Now we can show the main result of this section, namely that shearlet systems with generators satisfying
(\ref{eq:shearcond}) are actually instances of parabolic molecules associated with
a specific shearlet-adapted parametrization $\Phi_\sigma$.
\begin{prop}\label{prop:shearletmol}
    Assume that the shearlet system $\Sigma$ constitutes
    a system of shearlet molecules
    of order $(R,M,N_1,N_2)$.
    Then $\Sigma$ constitutes
    a system of parabolic molecules of order $(R,M,N_1,N_2)$, associated to the
    parametrization $(\Lambda^\sigma,\Phi^\sigma)$, where
    $$
        \Phi^\sigma(\lambda)=(s_\lambda,\theta_\lambda,x_\lambda):=
        \left(j,\varepsilon\pi/2+\arctan(-l2^{-\lfloor j/2\rfloor}),
        \left(S_l^\varepsilon\right)^{-1}
        D^\varepsilon_{2^{-j}}k\right).
    $$
\end{prop}
\begin{proof}
    We confine the discussion to $\varepsilon = 0$, the other case being the same.
    Further, we will suppress the superscript $\varepsilon$ as well as the subscript $j,l,k$ in
    our notation.
    We need to show that
    $$
        a^{(\lambda)}(\cdot):=\psi\left(D_{2^{s_\lambda}} S_{l,s_\lambda} R_{\theta_\lambda}^\top
        D_{2^{-s_\lambda}}\cdot\right)
    $$
    satisfies (\ref{eq:molcond}).
    The Fourier transform of $a^{(\lambda)}$ is
    given by
    $$
        \hat a^{(\lambda)}(\cdot) =\hat \psi \left(D_{2^{-s_\lambda}}
        S_{l,s_\lambda}^{-\top}R_{\theta_\lambda}^\top
        D_{2^{s_\lambda}}\cdot\right).
    $$
    The matrix $S_{l,s_\lambda}^{-\top}R_{\theta_\lambda}^\top$
    has the form
    $$
        S_{l,s_\lambda}^{-\top}R_{\theta_\lambda}^\top=\left(\begin{array}
        {cc}\cos(\theta_\lambda)&\sin(\theta_\lambda)\\0&
        -l2^{-\lfloor s_\lambda/2\rfloor}\sin(\theta_\lambda) +
         \cos(\theta_\lambda)\end{array}\right)=:
         \left(\begin{array}{cc} a & b\\0&c\end{array}\right).
    $$
    We claim that the quantities $a$ and $c$ are uniformly bounded from above
    and below, independent of $j,l$.
    To see this, consider the functions
    $$
        \tau(x):=\cos(\arctan(x)) \quad \mbox{and} \quad \rho(x):=x\sin(\arctan(x)) + \cos(\arctan(x)),
    $$
    which
    are bounded from above and below on $[-1,1]$, as an elementary discussion shows
    (in fact this boundedness holds on any compact interval).
    Clearly, we have
    $$
        a = \tau\left(-l2^{\lfloor \frac{s_\lambda}{2}\rfloor}\right)\quad \mbox{and} \quad
        c = \rho\left(-l2^{\lfloor \frac{s_\lambda}{2}\rfloor}\right)
    $$
    Since we are only considering indices with $\varepsilon = 0$, we have
    $\left|-l2^{\lfloor \frac{s_\lambda}{2}\rfloor}\right|\le 1$, which,
    by the above implies uniform upper and lower boundedness of the quantities
    $a,c$, i.e.,  there exist numbers $0<\delta_a\le \Delta_a <\infty$,
    $0<\delta_c\le \Delta_c <\infty$ such that for all
    $j,l$ we have
    \begin{equation*}
        \delta_a\le a\le \Delta_a\ \mbox{ and }\
        \delta_c\le c\le \Delta_c.
    \end{equation*}
    The matrix $D_{2^{-s_\lambda}}  R_{\theta_\lambda}^\top S_{l,s_\lambda}^{-\top}
    D_{2^{s_\lambda}}$ has the form
    $$
        \left(\begin{array}{cc}a & 2^{-s_\lambda/2}b \\0&c\end{array}\right).
    $$
    Using the upper boundedness of $a,b,c$ and the chain rule,
    we can estimate for any $|\beta |\le R$:
    $$
        |\partial^\beta \hat a^{(\lambda)}(\xi)|\lesssim \sup_{|\gamma|\le R}
        \left|\partial^\gamma
        \hat \psi\left(\left(\begin{array}{cc}a & 2^{-s_\lambda/2}b
         \\0&c\end{array}\right)\xi\right)\right|
        \lesssim \left( |\xi_1|+
        2^{-s_\lambda/2}|\xi_2|\right)^{M}.
    $$
    For the last estimate we utilized the moment estimate for $\hat\psi$,
    which is given by (\ref{eq:shearcond}).
    This gives us the moment property required in (\ref{eq:molcond}).

    Now we need to show the decay of $\partial^\beta \hat a^{(\lambda)}$
    for large frequencies $\xi$.
    Again, due to the fact that $a,b,c$ are bounded from above and $a,c$ from
    below, and utilizing the large frequency decay estimate in (\ref{eq:shearcond}),
    we can estimate
    \begin{eqnarray*}
        |\partial^\beta \hat a^{(\lambda)}(\xi)|
        &\lesssim &
        \sup_{|\gamma|\le R}
        \left|\partial^\gamma
        \hat \psi\left(\left(\begin{array}{cc}a & 2^{-s_\lambda/2}b
         \\0&c\end{array}\right)\xi\right)\right|
        \lesssim \left \langle\left|\left(\begin{array}{cc}a & 2^{-s_\lambda/2}b
         \\0&c\end{array}\right)\xi \right|
        \right\rangle^{-N_1} \langle c\xi_2 \rangle^{-N_2}
        \\
        &\lesssim &
        \left\langle |\xi|\right\rangle^{-N_1}\langle \xi_2\rangle^{-N_2}.
    \end{eqnarray*}
    The statement is proven.
\end{proof}
The following result shows that, just like the canonical parametrization, the shearlet
parametrization $\Lambda^\sigma$ is admissible.
\begin{lemma}
    The shearlet parametrization $(\Lambda^\sigma, \Phi^\sigma)$ is $k$-admissible
    for $k>2$.
\end{lemma}
\begin{proof}
    We show the analogue to Equation (\ref{eq:cd2}) for the shearlet parametrization, the rest of
    the proof is analogous to the proof of Lemma \ref{lem:canonadmin}.
    Hence, we aim to prove that
    \begin{equation}\label{eq:cd2shear}
        \sum_{\lambda \in \Lambda^\sigma, s_\lambda  = j}(1 + 2^{q}d(\mu,\lambda))^{-2}\lesssim
        2^{2(j-q)_+}
    \end{equation}
    for any $q$ and $\mu\in \Lambda^0$.
    Without loss of generality we assume that $\theta_{\mu}=0$, $x_\mu=0$
    (the general case follows identical arguments with slightly more notational effort).
    Further, as before
    we only restrict ourselves to the case $\varepsilon  = 0$, the other case being exactly
    the same.

    First we consider the case $q>j$. In this situation, the expression in (\ref{eq:cd2shear}) can be bounded
    by a uniform constant.

    Now we turn to the other case $j\geq q$. In this case we use the fact that,
    whenever $|l|\lesssim 2^{-j/2}$, we have
    $$
        \left|\mbox{arctan}\left(-l2^{-\lfloor \frac{j}{2}\rfloor}\right)\right|\gtrsim
        \left|l2^{-\lfloor \frac{j}{2}\rfloor}\right|
        \mbox{ and }|S_l^{-1}D_{2^{-j}}k| \gtrsim |D_{2^{-j}}k|,
    $$
    to estimate (\ref{eq:cd2shear}) by
    $$
        \sum_{l}\sum_{k}
        \left(1 + 2^q\left(\left|l2^{-\lfloor \frac{j}{2}\rfloor}\right|^2+
        \left|2^{-\lfloor \frac{j}{2}\rfloor}k_2\right|^2+
        \left|2^{-j}k_1 - l2^{-\lfloor \frac{j}{2}\rfloor}k_22^{-\lfloor \frac{j}{2}\rfloor}\right|\right)\right)^{-2}.
    $$
    This can be interpreted as a Riemann sum and bounded (up to a constant) by the corresponding integral
    $$
        \int_{\mathbb{R}^2}\frac{dx}{2^{-3j/2}}\int_\mathbb{R} \frac{dy}{2^{-j/2}}
        \left(1 + 2^q(y^2 + x_2^2 + |x_1 - x_2y|)\right)^{-2},
    $$
    compare \cite[Equation (A.3)]{CD02}.
    This integral is bounded by a constant times $2^{2(j-q)}$ as can be seen by making the
    substitution $x_1 \to 2^q x_1$, $x_2\to 2^{q/2}x_2$, $y\to 2^{q/2}y$.
    This shows (\ref{eq:cd2shear}) and thus completes the proof.
\end{proof}
These results show that the parabolic molecule concept is a unification of previous systems.
In the remainder of this section we examine the shearlet constructions which are on
the market and show that they indeed fit into our framework.
%
\subsubsection{Bandlimited Shearlets}
%
We start with the classical shearlet construction which yields bandlimited generators.
We consider two functions $\psi_1,\;\psi_2$ satisfying
$$
    \mbox{supp }\hat \psi_1 \subset \left [ -\frac{1}{2},-\frac{1}{16}\right]\cup
    \left [ \frac{1}{16},\frac{1}{2}\right],\quad
    \mbox{supp }\hat\psi_2 \subset \left[-1,1\right],
$$
$$
    \sum_{j\geq 0}\left|\hat\psi_1\left(2^{-j}\omega\right)\right|^2 = 1 \quad \mbox{for }
    |\omega| \geq \frac{1}{8},
$$
and
$$
    \sum_{l = -2^{\lfloor j/2\rfloor}}^{2^{\lfloor j/2\rfloor}}
    \left|\hat\psi_2\left(2^{\lfloor j/2\rfloor}\omega + l\right)\right|^2 = 1\quad
    \mbox{for }|\omega|\le 1.
$$
Now we define our basic shearlet $\psi^0$ via
$$
    \hat \psi^0(\xi) := \hat\psi_1(\xi_1)\hat\psi_2\left(\frac{\xi_2}{\xi_1}\right).
$$
It follows from standard arguments that the system
$$
    \Sigma^0:=\left\{2^{3j/4}\psi^0 \left(D^0_{2^j}
    S_{l,j}^0 \cdot - k\right):\;  j\geq 0,\ l  = -2^{\lfloor \frac{j}{2}\rfloor}, \cdots
    , 2^{\lfloor \frac{j}{2}\rfloor}\right\}
$$
constitutes a Parseval frame for the Hilbert space
$L_2\left(\mathcal{C}\right)^\lor$ with
$$
    \mathcal{C}:=\left\{\xi:\; |\xi_1|\geq \frac{1}{8},\;\frac{|\xi_2|}{|\xi_1|}\le 1\right\}.
$$
In the same way we can construct a Parseval frame $\Sigma^1$ for
$L_2\left(\mathcal{C}'\right)^\lor$,
$$
    \mathcal{C}':=\left\{\xi:\; |\xi_2|\geq \frac{1}{8},\;\frac{|\xi_1|}{|\xi_2|}\le 1\right\}.
$$
by reversing the coordinate axes.
Finally, we can consider a Parseval frame
$$
    \Phi:=\left\{\varphi(\cdot - k):\; k\in\mathbb{Z}^2\right\}
$$
for the Hilbert space $L_2\left(\left[-\frac{1}{8},\frac{1}{8}\right]^2\right)^\lor$.
\begin{prop}
    The system $\Sigma:=\Sigma^0\cup \Sigma^1\cup \Phi$ constitutes a
    shearlet frame which is a system of parabolic molecules of arbitrary order.
\end{prop}
\begin{proof}
    To show this, by Proposition \ref{prop:shearletmol}, all we need to show is
    that the generators $\psi^0,\; \psi^1$ satisfy (\ref{eq:shearcond})
    for arbitrary orders $(R,M,N_1,N_2)$. This, however, follows directly from
    the fact that the underlying basis functions are bandlimited.
\end{proof}
%
%
\subsubsection{Bandlimited Shearlets with Nice Duals}\label{sec:shearletswithduals}
%
The bandlimited shearlet frame $\Sigma$ as described above suffers from the
fact that we do not know much about its dual frames. In particular, we do not
know whether there exists a dual frame which is also a system of parabolic molecules.
For several results such as those in Subsection \ref{sec:functionspaces} it is
however necessary to have such a construction. In \cite{Grohs2011a} this problem
was successfully resolved by carefully glueing together the two bandlimited frames associated with
the two frequency cones. In other words, there exist shearlet frames $\Sigma$
with dual frame $\Sigma'$ such that both $\Sigma$ and $\Sigma'$ form systems
of parabolic molecules of arbitrary order.
%
\subsubsection{Smooth Parseval Frames of Shearlets}
%
In \cite{Guo2012a} another modification of the bandlimited shearlet construction
is given by carefully glueing together two boundary elements along
the seamlines with angle $\pi/4$. It can be shown that this yields
a Parseval frame with smooth and well-localized elements.
Again, it is straightforward to check that the system constructed in \cite{Guo2012a}
constitutes a system of parabolic molecules of arbitrary order.
%
\subsubsection{Compactly Supported Shearlets}
%
We next analyze compactly supported shearlets \cite{Kittipoom2010}, and prove that they also form instances of parabolic molecules.
Currently known constructions of compactly supported shearlets involve separable
generators, i.e.,
\begin{equation}\label{eq:shearsep}
    \psi^0(x_1,x_2):=\psi_1(x_1)\psi_2(x_2),\quad \psi^1(x_1,x_2):=\psi^0(x_2,x_1).
\end{equation}
with a wavelet $\psi_1$ and a scaling function $\psi_2$.
We would like to find conditions on $\psi_1,\psi_2$ such that
(\ref{eq:shearcond}) is satisfied for given parameters $(R,M,N_1,N_2)$,
i.e., that the associated shearlet frame forms a system of shearlet molecules.

First we define the crucial property of vanishing moments for univariate
wavelets.
\begin{definition}\label{def:vanishingmoments}
    A univariate function $g$ possesses $M$ vanishing moments
    if
    $$
        \int_{\mathbb{R}} g(x) x^k dx = 0,\quad \mbox{for all }k=0,\dots , M-1.
    $$
\end{definition}
In the frequency domain, vanishing moments are characterized
by polynomial decay near zero, as is well known.
\begin{lemma}\label{lem:momentsfreq}
    Suppose that $g:\mathbb{R}\to \mathbb{C}$
    is continuous, compactly supported and possesses $M$ vanishing moments.
    Then
    $$
        |\hat g(\xi)|\lesssim \min\left(1,|\xi|\right)^M.
    $$
\end{lemma}
\begin{proof}
    First, note that, since $g$ is continuous and compactly supported, it
    is in $L_1(\mathbb{R})$ and therefore its Fourier transform is bounded.
    This takes care of frequencies $\xi$ with $|\xi|\geq 1$.
    For small $\xi$ observe that, up to a constant
    we have
    $$
        \int_{\mathbb{R}} g(x) x^k dx = \left(\frac{d}{d\xi}\right)^k\hat g(0).
    $$
    Hence, if $g$ possesses $M$ vanishing moments, all
    derivatives of order $<M$ of the Fourier transform $\hat g$ vanish
    at $0$. Furthermore, since $g$ is compactly supported, its Fourier transform
    is analytic. Therefore
    $$
        |\hat g(\xi)|\lesssim |\xi|^M,
    $$
    which proves the claim.
\end{proof}
\begin{prop}
    Assume that $\psi_1\in C^{N_1}$ is a compactly supported wavelet with $M+R$ vanishing moments,
    and $\psi_2\in C^{N_1+N_2}$ is also compactly supported.
    Then, with $\psi^\varepsilon$ defined by (\ref{eq:shearsep}), the
    associated shearlet system
    $\Sigma$ constitutes a system of parabolic molecules of order $(R,M,N_1,N_2)$.
\end{prop}
\begin{proof}
    In view of Proposition \ref{prop:shearletmol} we need to show that
    the estimate (\ref{eq:shearcond}) holds.
    We only consider the case $\varepsilon = 0$ and drop the superscript.
    The inverse Fourier transform of $\partial^\beta \psi$ is, up
    to a constant given by $x^\beta \psi(x)$.
    We first handle the case $|\xi_1|>1$.
    By smoothness and compact support of $\psi_1,\psi_2$,
    we find that for any $|\beta | \le R$ the function
    $$
        \partial^{(N_1,N_1+N_2)} x^\beta \psi
    $$
    is in $L_1(\mathbb{R})$, hence it has a bounded Fourier transform
    which is given, up to a constant by
    $$
        \xi_1^{N_1} \xi_2^{N_1+N_2}\partial^\beta \hat \psi(\xi).
    $$
    It follows that the function
    $$
        \langle \xi_1 \rangle^{N_1}\langle \xi_2 \rangle^{N_1+N_2}
        \partial^\beta \hat \psi(\xi)
    $$
    is bounded.
    Using the simple fact that
    $\langle x \rangle \langle y \rangle \lesssim \langle \sqrt{x^2+y^2}\rangle$,
    we get
    $$
        \partial^\beta\psi(\xi) \lesssim
        \langle \| \xi\|\rangle^{-N_1}\langle \xi_2\rangle^{-N_2}.
    $$
    Now let $\beta$ be such that $|\beta_1|<R$.
    Then
    the function
    $$
        x^\beta \psi(x)= x_1^{\beta_1}x_2^{\beta_2}\psi_1(x_1)\psi_2(x_2),
    $$
    restricted to the variable $x_1$ possesses at least
    $M$ vanishing moments, due to the assumption that $\psi_1$ possesses
    $M+R$ vanishing moments. Lemma \ref{lem:momentsfreq} then proves
    the decay of order $\min\left(1,|\xi_1|^M\right)$ for the
    derivatives of $\hat \psi$.
\end{proof}
\begin{remark}
    Several assumptions on the generators $\psi^1,\ \psi^2$
    could be weakened, for instance the separability of the
    shearlet generators is not crucial for the arguments to go through.
    In particular, our arguments nowhere require neither compact
    support nor bandlimitedness.
\end{remark}
%
%
\subsubsection{Shearlet Molecules of \cite{GL08}}
\label{subsec:shearletmole}
%
In \cite{GL08} the results of \cite{CD02} are established for
shearlets instead of curvelets. A crucial tool in the proof is the
introduction of a certain type of \emph{shearlet molecules} which are
similar to curvelet molecules discussed above, but tailored to the
shearing operation rather than rotations.
\begin{definition}
    Let $\Lambda^\sigma$ be the shearlet index set as in (\ref{eq:shearindex}).
    A family $(m_\lambda)_{\lambda \in \Lambda^\sigma}$ is
    called a \emph{family of shearlet molecules} of
    regularity $R$ if
    it can be written as
    $$
        m_\lambda (x) =
        2^{3s_\lambda/4}
        a^{(\lambda)}
        \left(D_{2^{s_\lambda}}^\varepsilon S_{l,j}^\varepsilon x - k \right)
    $$
    such that for all $|\beta| \le R$ and each $N = 0,1,2,\dots$
    \begin{equation*}
        |\partial^\beta a^{(\lambda)}(x)|\lesssim \langle x\rangle^{-N}
    \end{equation*}
    and for $M = 0,1,\dots $
    \begin{equation*}
        |\hat{a}^{(\lambda)}(\xi)|\lesssim
        \min\left(1,2^{-s_\lambda} + |\xi_1| + 2^{-s_\lambda/2}|\xi_2|\right)^M.
    \end{equation*}
\end{definition}
By the results in \cite{GL08}, the shearlet molecules defined
therein satisfy the inequality (\ref{eq:shearcond}) with the choice
of parameters $(R,N,N_1,N_2) = (\infty,\infty, R/2,R/2)$. Therefore, in view
of Proposition \ref{prop:shearletmol}, shearlet molecules of
regularity $R$ as defined in \cite{GL08} form systems of
parabolic molecules of order $(\infty,\infty,R/2,R/2)$. Thus, we derive
an analogous result to Lemma \ref{lem:curvemol} for shearlet molecules:

\begin{prop}
    A system of shearlet molecules of regularity $R$ constitutes a system of
    parabolic molecules of order $(\infty, \infty , R/2, R/2)$.
\end{prop}

Let us finish this section on examples of parabolic molecules by making the following
\begin{remark}
    By now we hope to have convinced the reader that there is a whole
    zoo of different constructions in the literature which can all
    be put under one roof using the concept of parabolic molecules.
\end{remark}
%

%
\section{Applications}
\label{sec:appl}
%
In this section we discuss selected applications of the developed theory.
A particular focus will be on approximation properties of parabolic molecule
systems
$(m_\lambda)_{\lambda\in \Lambda}$, especially
if they form a frame, e.g.,
$$
    \|f\|_{L_2(\mathbb{R}^2)}\sim \sum_{\lambda\in \Lambda}|\langle f, m_\lambda \rangle|^2,
$$
see e.g., \cite{Christensen2003a}. It is well known, that in this case one can
robustly represent any function $f\in L_2(\mathbb{R}^2)$ as a sum
$$
    f = \sum_{\lambda\in \Lambda}\langle f , m_\lambda\rangle \tilde m_\lambda,
$$
where $(\tilde m_\lambda)_{\lambda\in \Lambda}$ is a \emph{dual frame}.
Approximation properties of a frame system $(m_\lambda)_{\lambda\in \Lambda}$
are usually studied in terms of the sparsity of the coefficient sequence
$
    \left(\langle f , m_\lambda\rangle\right)_{\lambda\in \Lambda}
$.
Below, in Subsection \ref{subsec:sparse} we show that essentially any
frame system which consists of parabolic molecules satisfies the same
approximation properties as the curvelet frame constructed in \cite{Candes2005a}.
This, in particular, implies almost optimal approximation results for the class
of cartoon images (see below for a definition) for all constructions mentioned in Section
\ref{sec:exam}, for instance compactly supported or bandlimited shearlets.
The above-mentioned approximation property may actually be regarded as the main
raison d'\`etre of curvelet-like systems and is therefore of central importance.

In Subsection \ref{sec:functionspaces} we go further and show that practically any reasonable
definition of a function space norm based on a coefficient sequence
$
    \left(\langle f , m_\lambda\rangle\right)_{\lambda\in \Lambda}
$
is equivalent for any two frame systems consisting of parabolic molecules.
This shows for instance that finiteness of a function space norm defined via the
curvelet frame implies finiteness of the analogous norm defined via compactly supported
shearlet frames, whenever the generators possess sufficient smoothness and directional vanishing
moments. This result has not been known before.
\begin{remark}
    It is in general not the case that the dual frame $(\tilde m_\lambda)_{\lambda\in \Lambda}$ of
    a frame $(m_\lambda)_{\lambda\in \Lambda}$ of parabolic molecules needs to consist of parabolic molecules, too.
    However, it can be shown, based on the concept of \emph{intrinsic localization}, that
    the so-called canonical dual frame of $(m_\lambda)_{\lambda\in \Lambda}$ is of a similar form in
    a certain sense \cite{Grohs2012}.
\end{remark}
%
%
\subsection{Sparse Image Approximation}
\label{subsec:sparse}

Multivariate problems are typically governed by anisotropic features such as edges in images. A customarily
employed model for such data is the class $\cE^2(\RR^2)$ of so-called {\em cartoon images} which is
defined by
\[
\cE^2(\RR^2) = \{f \in L^2(\RR^2) : f = f_0 + f_1 \cdot \chi_{B}\},
\]
where $B \subset [0,1]^2$ with $\partial B$ a closed $C^2$-curve and $f_0,f_1 \in C_0^2([0,1]^2)$. Questions
of efficient encoding of such a model class can be formulated in terms of optimal approximation properties.
Given a frame system $(m_\lambda)_\lambda \subseteq L_2(\RR^2)$, an appropriate measure for
the approximation behavior is the decay rate of the {\em error of best $N$-term approximation}, i.e., of
$\norm{f - f_N}_2^2$, where $f_N$ denotes the best $N$-term approximation by $(m_\lambda)_\lambda$
of some $f \in \cE^2(\RR^2)$, obtained as
$$
    f_N = \argmin \|f -  \sum_{\lambda \in \Lambda_N}c_\lambda m_\lambda\|_2^2 \quad
    \mbox{s.t.}\quad \#\Lambda_N \le N.
$$
A small technical problem occurs due to the fact that the representation
system might not form an orthonormal basis in which case the computation of the best $N$-term approximation
is far from being understood. To circumvent this problem, usually the error of approximation by the $N$
largest coefficients of $\left(\langle f , m_\lambda\rangle\right)_{\lambda\in \Lambda}$
is considered, which then certainly also provides a bound for the error of best $N$-term
approximation.
Typically, the asymptotics of this error are studied in terms of the $\ell_p$-norms of the coefficient sequences
of $f$ for small values of $p$. Indeed, it is easily seen that membership of the coefficient sequence
of $f$ in an $\ell_p$ space for small $p$ implies good $N$-term approximation rates, whenever the given
representation system constitutes a frame, see, e.g., \cite{Kutyniok2010,Devore1998}.

In \cite{Don01} it was shown that the optimally achievable decay rate of the error of approximation of some
$f \in \cE^2(\RR^2)$ under the natural assumption of polynomial depth search is
\[
\norm{f - f_N}_2^2 \asymp N^{-2}, \quad \mbox{as } N \to \infty.
\]
Furthermore, it was proven in \cite{CD04} and in \cite{GL07,Kutyniok2010} that both curvelets and shearlets
attain this rate up to a log-factor. Apparently, these (parabolic) systems behave similarly concerning
sparse approximation of anisotropic features.

The next definition provides a formalization of this concept by introducing the notion of sparsity equivalence.
It is based on the connection between best $N$-term approximation rate and $\ell_p$ norms.
\begin{definition}
Let $(m_\lambda)_{\lambda\in \Lambda}$ and $(p_\mu)_{\mu \in M}$ be systems of parabolic molecules of order
$(R,M,N_1,N_2)$ and $(\tilde{R},\tilde{M},\tilde{N}_1,\tilde{N}_2)$, respectively, and let $0 < p \le 1$.
Then $(m_\lambda)_{\lambda\in \Lambda}$ and $(p_\mu)_{\mu \in M}$ are {\em sparsity equivalent in $\ell_p$},
if
\[
\left\|\left(\langle m_{\lambda},p_\mu\rangle\right)_{\lambda\in \Lambda, \mu\in \Lambda^0}\right\|_{\ell_p\to \ell_p}<\infty.
\]
\end{definition}

Intuitively, systems of parabolic molecules being in the same sparsity equivalence class should have the same
sparse approximation behavior with respect to cartoon images. The next result shows that this is indeed the
case.

\begin{prop}
\label{prop:sparsityequiv}
Let  $(m_\lambda)_{\lambda\in \Lambda}$ and $(p_\mu)_{\mu \in M}$ be systems of parabolic molecules of order
$(R,M,N_1,N_2)$ and $(\tilde{R},\tilde{M},\tilde{N}_1,\tilde{N}_2)$, respectively, which are sparsity equivalent
in $\ell_{2/3}$. If $(m_\lambda)_{\lambda\in \Lambda}$ possesses an almost best $N$-term approximation rate of
order $N^{-1+\varepsilon}$ for cartoon images for any $\varepsilon >0$, then so does $(p_\mu)_{\mu \in M}$.
\end{prop}

\begin{proof}
    This is a direct consequence of the definition of sparsity equivalence and standard arguments, see for instance
    \cite{Devore1998}.
\end{proof}

This result enables us to provide a very general class of systems of parabolic molecules which optimally sparsely
approximate cartoon images by using the known result for curvelets. For this, we first analyze when a system is
sparsity equivalent to the tight frame of bandlimited curvelets.

First, we state a simple result concerning operator norms on discrete $\ell_p$ spaces.
\begin{lemma}\label{lem:matnorm}
    Let $I,J$ be two discrete index sets, and let ${\bf A} : \ell_p(I)\to \ell_p(J)$, $p>0$ be a linear mapping
    defined by its matrix representation
    ${\bf A} = \left(A_{i,j}\right)_{i\in I ,\, j\in J}$.
    Then we have the bound
    \begin{equation*}
        \|{\bf A}\|_{\ell_p(I)\to \ell_p(J)}\le
        \max\left(\sup_{i}\sum_{j}|A_{i,j}|^{r},\sup_{j}\sum_{i}|A_{i,j}|^r\right)^{1/r},
    \end{equation*}
    where $r:=\min(1,p)$.
\end{lemma}
\begin{proof}
    The proof for $p<1$ follows easily using the fact that
    $$
        |a + b|^p\le |a|^p + |b|^p\quad \mbox{ for }a,b\in\mathbb{R}.
    $$
    To show the case $p\geq 1$ one only shows the assertion for $p=1,\infty$,
    which is trivial. The claim then follows by interpolation.
\end{proof}
The next theorem proves the central fact that \emph{any} system of parabolic molecules of sufficiently
high order is sparsity equivalent to the bandlimited curvelet frame from Subsection \ref{sec:curvelets}.
\begin{theorem}
\label{theo:curveletothersystem}
    Assume that $0<p\le 1$, $(\Lambda,\Phi_\Lambda)$ is a $k$-admissible parametrization,
    and $\Gamma^0 = (\gamma_\lambda)_{\lambda\in \Lambda^0}$ the tight
    frame of bandlimited curvelets. Further, assume that $(m_\lambda)_{\lambda\in \Lambda}$
    is a system of molecules associated with $\Lambda$ of order $(R,M,N_1,N_2)$
    such that
    $$
        R\geq 2\frac{k}{p},\quad  M> 4\frac{k}{p} - \frac{5}{4} ,\quad N_1 \geq 2\frac{k}{p}+\frac{3}{4} , \quad
        N_2\geq 2\frac{k}{p}.
    $$
    Then $(m_\lambda)_{\lambda\in \Lambda}$ is sparsity equivalent to $\Gamma^0$.
\end{theorem}
\begin{proof}
    We need to show that
    $$
        \left\|\left(\langle m_{\lambda},\gamma_\mu\rangle\right)_{\lambda\in \Lambda,
        \mu\in \Lambda^0}\right\|_{\ell_p\to \ell_p}=
        \max \left( \sup_{\mu\in\Lambda}\sum_{\lambda\in \Lambda^0}
        |\langle m_{\lambda},\gamma_\mu\rangle|^p,
        \sup_{\lambda\in\Lambda^0}\sum_{\mu\in \Lambda}
        |\langle m_{\lambda},\gamma_\mu\rangle|^p\right)^{1/p} <\infty.
    $$
    By Theorem \ref{thm:almostorth}, we have
    $$
        |\langle m_{\lambda},\gamma_\mu\rangle|\lesssim \omega(\lambda,\mu)^{-\frac{k}{p}}.
    $$
    It follows that
    \begin{eqnarray*}
        \lefteqn{\max \left( \sup_{\mu\in\Lambda}\sum_{\lambda\in \Lambda^0}
        |\langle m_{\lambda},\gamma_\mu\rangle|^p,
        \sup_{\lambda\in\Lambda^0}\sum_{\mu\in \Lambda^0}
        |\langle m_{\lambda},\gamma_\mu\rangle|^p\right)}\\
        & \lesssim &
        \max \left( \sup_{\mu\in\Lambda}\sum_{\lambda\in \Lambda^0}
        \omega(\lambda,\mu)^{-k},
        \sup_{\lambda\in\Lambda^0}\sum_{\mu\in \Lambda^0}
        \omega(\lambda,\mu)^{-k}\right)<\infty,
    \end{eqnarray*}
    due to the $k$-admissibility of the parametrization of $\Lambda$.
\end{proof}
This result in combination with Proposition \ref{prop:sparsityequiv} now leads to the main result of
this subsection.
\begin{theorem}
\label{theo:mainsparsity}
    Assume that $(m_\lambda)_{\lambda\in \Lambda}$ is a system of parabolic molecules
    of order $(R,M,N_1,N_2)$ such that
    \begin{itemize}
    \item[(i)] $(m_\lambda)_{\lambda\in \Lambda}$ constitutes a frame for $L_2(\mathbb{R}^2)$,
    \item[(ii)] $\Lambda$ is $k$-admissible for all $k>2$,
    \item[(iii)] it holds that
        $$
            R\geq 6,\quad  M > 12 - \frac{5}{4} ,\quad N_1 \geq 6+\frac{3}{4} , \quad
            N_2\geq 6.
        $$
    \end{itemize}
    Then the frame $(m_\lambda)_{\lambda\in \Lambda}$ possesses an almost best $N$-term
    approximation rate of order $N^{-1+\varepsilon}$, $\varepsilon >0$ arbitrary
    for cartoon images.
\end{theorem}
\begin{proof}
    This follows from Proposition \ref{prop:sparsityequiv}, Theorem \ref{theo:curveletothersystem},
    and the fact, proven in \cite{CD04}, that $\Gamma^0$ provides the respective $N$-term approximation rate.
\end{proof}

We remark that condition (ii) holds in particular for the shearlet parametrization. Hence
this result allows a simple derivation of the results in \cite{GL07,Kutyniok2010}
from \cite{CD04}. In fact, Theorem \ref{theo:mainsparsity} provides a systematic way to, in particular,
prove results on sparse approximation of cartoon images.

%
\subsection{Function Spaces}
\label{sec:functionspaces}
%
Based on the concept of decomposition spaces introduced in \cite{Feichtinger1985}, Borup
and Nielsen have studied curvelet-like function spaces in \cite{Borup2007}.
We would like to apply our results to show that these spaces can be characterized
by the transform coefficients in any frame which also forms a system of parabolic molecules.
Consider the curvelet frame
$$
    \Gamma^0:=\left\{\gamma_{j,l,k}: (j,l,k) \in \Lambda^0 \right\}
$$
introduced in Subsection \ref{sec:curvelets}.
Following \cite{Borup2007} we define for $p,q,\alpha >1$
the function spaces $G_{p,q}^\alpha$ given by
the norm
\begin{equation}\label{eq:gnorm}
    \|f\|_{G_{p,q}^\alpha}:=\left(\sum_{j\geq 0,l}\left( 2^{\alpha j}
    \left(\sum_{k}|\langle
    f , \gamma_{j,l,k}\rangle|^p\right)^{1/p}\right)^q\right)^{1/q}.
\end{equation}
This definition might seem somewhat odd, since the summation with respect to
the directional parameter $l$ is done with respect to the $\ell_q$ norm.
For this reason and also for some minor technical reasons,
we study another, similar family of function spaces, namely the
spaces $S_{p,q}^\alpha$ given by the norm
\begin{equation}\label{eq:snorm}
    \|f\|_{S_{p,q}^\alpha}:=\left(\sum_{j\geq 0}\left( 2^{\alpha j}
    \left(\sum_{k,l}|\langle
    f , \gamma_{j,l,k}\rangle|^p\right)^{1/p}\right)^q\right)^{1/q}.
\end{equation}
\begin{remark}
    We would like to emphazise
    that all the results shown in this section also hold for the spaces defined
    by (\ref{eq:gnorm}), but with slightly more technical effort arising from
    the need to handle mixed Lebesgue spaces \cite{Benedek1961}.
    We also remark that the function spaces defined via (\ref{eq:snorm})
    can be interpreted as a decomposition spaces
    of the form studied in \cite{Borup2007} with a mixed Lebesgue space $Y=\ell_q \ell_p$
    (see \cite{Borup2007} for more information).
\end{remark}
For technical reasons the definition in (\ref{eq:snorm}) forces us to work with a slightly
stronger notion of admissibility than given in Definition \ref{def:kadmiss}:
\begin{definition}
    An index set $\Lambda$ with associated mapping $\Phi_\Lambda$ is called
    \emph{strongly $(k,l)$-admissible} if it is $k$-admissible and
    if
    \begin{equation*}
        \sum_{\lambda \in \Lambda_j}(1 + 2^{q}d(\mu,\lambda))^{-k}\lesssim
        2^{l(j-q)_+},
    \end{equation*}
    where
    $$
        \Lambda_j:=\left\{\lambda\in \Lambda:\; s_\lambda = j \right\}.
    $$
\end{definition}
\begin{lemma}
    The canonical parametrization $(\Lambda^0,\Phi^0)$ and the shearlet parametrization
    $(\Lambda^\sigma,\Phi^\sigma)$ are both strongly $(k,2)$-admissible for any $k>2$.
\end{lemma}
\begin{proof}
    This has already been shown earlier in (\ref{eq:cd2}) for the canonical
    parametrization and in (\ref{eq:cd2shear}) for the shearlet parametrization.
\end{proof}
The aim of this section is to show the following theorem.
\begin{theorem}\label{thm:normequiv}
    Let $\Sigma=\left\{\sigma_\lambda:\;\lambda\in \Lambda\right\}$
    be a frame for $L_2(\mathbb{R}^2)$ with dual frame $\tilde \Sigma=
    \left\{\tilde \sigma_\lambda:\;\lambda\in \Lambda\right\}$.
    Assume further that $\Sigma$, $\tilde \Sigma$ are both parabolic molecules of arbitrary
    order with a strongly $(k,l)$ admissible parametrization for some $k,l$.
    Then the following are equivalent norms on
    $S_{p.q}^\alpha$:
    \begin{equation*}
        \|f\|_{S_{p,q}^\alpha}\sim \left(\sum_{j\geq 0}\left( 2^{\alpha j}
        \left(\sum_{\lambda\in\Lambda_j}|\langle
        f , \sigma_{\lambda}\rangle|^p\right)^{1/p}\right)^q\right)^{1/q}
        \sim
        \left(\sum_{j\geq 0}\left( 2^{\alpha j}
        \left(\sum_{\lambda\in\Lambda_j}|\langle
        f , \tilde \sigma_{\lambda}\rangle|^p\right)^{1/p}\right)^q\right)^{1/q}.
    \end{equation*}
\end{theorem}
\begin{remark}
    Of course it would be possible to show a quantitative version of Theorem \ref{thm:normequiv}
    in the sense that $\Sigma$ and $\tilde \Sigma$ are only required to
    form a system of parabolic molecules of finite, sufficiently large
    order, depending on $p,q,\alpha$.
\end{remark}
Before we start with the proof of
Theorem \ref{thm:normequiv} we recall the following result which
is a very useful inequality, sometimes called
the \emph{discrete Hardy inequality}, see \cite{Devore1988}.
To state this result we define for a sequence ${\bf a}= \left(a_k\right)_{k\in\mathbb{N}}$
the (quasi) norm
\begin{equation*}
    \|{\bf a}\|_{\ell_q^\alpha}:=\left(\sum_{k\in\mathbb{N}}\left(2^{k\alpha}|a_k|\right)^q\right)^{1/q}.
\end{equation*}
The discrete Hardy inequalities are as follows.
\begin{lemma}\label{lem:hardy}
    Assume that with $\lambda > \alpha$ and $r\le q$ we have that either
    \begin{equation}\label{eq:hardy0}
        |b_k|\lesssim 2^{-\lambda k}\left(\sum_{j=0}^k\left(2^{\lambda j}|a_j|\right)^r\right)^{1/r},
    \end{equation}
    or
    \begin{equation}\label{eq:hardy1}
        |b_k|\lesssim \left(\sum_{j=k}^\infty |a_j|^r\right)^{1/r}.
    \end{equation}
    Then we have
    \begin{equation*}
        \|{\bf b}\|_{\ell_q^\alpha}\lesssim \|{\bf a}\|_{\ell_q^\alpha}.
    \end{equation*}
\end{lemma}
Observe that defining $a_k:= \left\|\left(\langle f , \gamma_\lambda \rangle \right)_{\lambda\in\Lambda_k}\right\|_p$,
we have $\|f\|_{S_{p,q}^\alpha} = \|{\bf a}\|_{\ell_q^\alpha}$.
Armed with these useful facts, we may now proceed with the proof of Theorem \ref{thm:normequiv}.
\begin{proof}[Proof of Theorem \ref{thm:normequiv}]
    We start by fixing some notation:
    $$
        {\bf f }^\Gamma := \left(\langle f , \gamma_\mu \rangle \right)_{\mu\in\Lambda^0},
        \quad
        {\bf f }^\Gamma_j := \left(\langle f , \gamma_\mu \rangle \right)_{\mu\in\Lambda^0_j},
        \quad
        {\bf f }^\Sigma := \left(\langle f , \sigma_\lambda \rangle \right)_{\lambda\in\Lambda},
        \quad
        {\bf f }^\Sigma_j := \left(\langle f , \sigma_\lambda \rangle \right)_{\lambda\in\Lambda_j}.
    $$
    Further, we write $\Gamma_j=\left\{\gamma_\mu:\;\mu\in \Lambda^0_j\right\}$ and
    similar for the systems $\Sigma,\,\tilde \Sigma$. Define
    $$
        {\bf A}:=\langle \Gamma, \Sigma\rangle,\quad {\bf A}_{i,j}
        :=\langle \Gamma_i, \Sigma_j\rangle,\quad
        {\bf \tilde A}:=\langle \tilde \Sigma, \Gamma\rangle,\quad {\bf A}_{i,j}
        :=\langle \tilde \Sigma_i, \Gamma_j\rangle.
    $$
    We have
    $$
        {\bf f}^\Sigma = \left({\bf f}^\Gamma\right)^\top {\bf A},\quad
        {\bf f}^\Sigma_i = \sum_{j\geq 0}\left({\bf f}^\Gamma_j\right)^\top {\bf A}_{i,j},\quad
        {\bf f}^\Gamma = \left({\bf f}^\Sigma\right)^\top {\bf A},\quad
        {\bf f}^\Gamma_i = \sum_{j\geq 0}\left({\bf f}^\Sigma_j\right)^\top {\bf A}_{i,j}.
    $$
    Let us first assume that $p\geq 1$. Then we would like to show that
    $$
        \left(\sum_{j\geq 0}\left( 2^{\alpha j}
        \left(\sum_{\lambda\in\Lambda_j}|\langle
        f , \sigma_{\lambda}\rangle|^p\right)^{1/p}\right)^q\right)^{1/q}<\infty,
    $$
    whenever
    $$
        \left(\sum_{j\geq 0}\left( 2^{\alpha j}
        \left(\sum_{\mu\in\Lambda^0_j}|\langle
        f , \gamma_{\mu}\rangle|^p\right)^{1/p}\right)^q\right)^{1/q}<\infty.
    $$
    For this, we obtain
    \begin{equation}\label{eq:triangle}
        b_i:=\left\|{\bf f}^\Sigma_i\right\|_p =
        \left\|\sum_{j\geq 0}\left({\bf f}^\Gamma_j\right)^\top {\bf A}_{i,j}\right\|_p
        \le \sum_{j\geq 0}\left\|\left({\bf f}^\Gamma_j\right)^\top {\bf A}_{i,j}\right\|_p
        = d_i + e_i,
    \end{equation}
    where
    \begin{equation*}
        d_i:= \sum_{j>i}\left\|\left({\bf f}^\Gamma_j\right)^\top {\bf A}_{i,j}\right\|_p
        \quad\mbox{and}\quad
        e_i:= \sum_{j\le i}\left\|\left({\bf f}^\Gamma_j\right)^\top {\bf A}_{i,j}\right\|_p.
    \end{equation*}
    Next, we will prove that the inequalities (\ref{eq:hardy0}), (\ref{eq:hardy1}) are satisfied
    for the sequences $d_i$, $e_i$, respectively. By Lemma \ref{lem:hardy}, this proves the
    desired claim.

    We start by deriving the following estimate for $d_i$:
    \begin{equation*}
        d_i \le
        \sum_{j>i}\left\|\left({\bf f}^\Gamma_j\right)\|_p\| {\bf A}_{i,j}\right\|_{\ell_p\to \ell_p}
    \end{equation*}
   To further analyze $\left\| {\bf A}_{i,j}\right\|_{\ell_p\to \ell_p}$, we employ Lemma
    \ref{lem:matnorm} to obtain
    \begin{equation}\label{eq:lpnorm}
        \left\| {\bf A}_{i,j}\right\|_{\ell_p\to \ell_p} \le \max\left(\sup_{\mu\in \Lambda^0_i}\sum_{\lambda\in\Lambda_j}
        |\langle \gamma_\mu,\sigma_\lambda\rangle|,
        \sup_{\lambda\in\Lambda_j}\sum_{\mu\in \Lambda^0_i}
        |\langle \gamma_\mu,\sigma_\lambda\rangle|
        \right)
    \end{equation}
    Using the fact that $\Gamma$ and $\Sigma$ are parabolic molecules of arbitrary order,
    Theorem \ref{thm:almostorth} implies that for $N$ arbitrary,
    \begin{equation}\label{eq:diag}
        |\langle \gamma_\mu,\sigma_\lambda\rangle|\lesssim \omega(\mu,\lambda)^{-N}.
    \end{equation}
    By (\ref{eq:diag}) and the fact that the parametrization for $\Lambda$
    is strongly admissible, we can further estimate the first term in (\ref{eq:lpnorm}) by
    \begin{equation*}
        \sup_{\mu\in \Lambda^0_i}\sum_{\lambda\in\Lambda_j}
        \omega(\mu,\lambda)^{-N} =
        2^{-N|i-j|}\sup_{\mu\in \Lambda^0_i}\sum_{\lambda\in\Lambda_j}
        \left(1 + 2^{\min(i,j)}d(\mu,\lambda)\right)^{-N}
        \lesssim 2^{-(N-l)|i-j|}.
    \end{equation*}
    The second term is treated similarly, and we wind up with
    \begin{equation}\label{eq:matrixnormdecay}
        \left\| {\bf A}_{i,j}\right\|_{\ell_p\to \ell_p}\lesssim 2^{-N|i-j|}
    \end{equation}
    for $N$ arbitrarily large.
    In particular, this implies that
    \begin{equation*}
        d_i\lesssim
        \sum_{j>i}\left\|\left({\bf f}^\Gamma_j\right)\right\|_p\lesssim
        \left(\sum_{j>i}\left\|\left({\bf f}^\Gamma_j\right)\right\|_p\right)^{1/r}
    \end{equation*}
    with $r:=\min(1,q)$, and this is (\ref{eq:hardy1}).
    Similarly we can estimate
    \begin{equation*}
        e_i \lesssim \sum_{j\le i}2^{-N(i-j)}
        \left\|\left({\bf f}^\Gamma_j\right)\right\|_p
        = 2^{-N i}\sum_{j\le i}2^{N j}
        \left\|\left({\bf f}^\Gamma_j\right)\right\|_p
        \lesssim \
        2^{-N i}\left(\sum_{j\le i}\left(2^{N j}
        \left\|\left({\bf f}^\Gamma_j\right)\right\|_p\right)^r\right)^{1/r}
    \end{equation*}
    which is (\ref{eq:hardy0}). Applying Lemma \ref{lem:hardy} yields
    $$
        \left(\sum_{j\geq 0}\left( 2^{\alpha j}
        \left(\sum_{\lambda\in\Lambda_j}|\langle
        f , \sigma_{\lambda}\rangle|^p\right)^{1/p}\right)^q\right)^{1/q}
        \lesssim
        \left(\sum_{j\geq 0}\left( 2^{\alpha j}
        \left(\sum_{\mu\in\Lambda^0_j}|\langle
        f , \gamma_{\mu}\rangle|^p\right)^{1/p}\right)^q\right)^{1/q}
    $$
    which proves one half of the desired norm equivalence. The other half (and the case
    of $\tilde \Sigma$) can be shown in exactly the same way. Therefore, for $p\geq 1$, the claim of the theorem
    is proven.

    Let us now turn to the case $p<1$. For this, we need to replace the estimate (\ref{eq:triangle})
    with
    {\allowdisplaybreaks
    \begin{eqnarray*}
        |b_i| &:= &\left\|{\bf f}^\Sigma_i\right\|_p =
        \left\|\sum_{j\geq 0}\left({\bf f}^\Gamma_j\right)^\top {\bf A}_{i,j}\right\|_p
        \\
        & \lesssim &
        \left\|\sum_{j\le i}\left({\bf f}^\Gamma_j\right)^\top {\bf A}_{i,j}\right\|_p
        +\left\|\sum_{j > i}\left({\bf f}^\Gamma_j\right)^\top {\bf A}_{i,j}\right\|_p
        \\
        &\le &
        \left(\sum_{j\le i}\left\|{\bf f}^\Gamma_j\right\|_p^p
        \left\|{\bf A}_{i,j}\right\|_{\ell_p\to \ell_p}^p\right)^{1/p} +
        \left(\sum_{j > i}\left\|{\bf f}^\Gamma_j\right\|_p^p
        \left\|{\bf A}_{i,j}\right\|_{\ell_p\to \ell_p}^p\right)^{1/p}
        \\
        &\lesssim&
        \left(\sum_{j\le i}\left\|{\bf f}^\Gamma_j\right\|_p^r
        \left\|{\bf A}_{i,j}\right\|_{\ell_p\to \ell_p}^r\right)^{1/r} +
        \left(\sum_{j > i}\left\|{\bf f}^\Gamma_j\right\|_p^r
        \left\|{\bf A}_{i,j}\right\|_{\ell_p\to \ell_p}^r\right)^{1/r}
        \\
        &=:& d_i + e_i,
    \end{eqnarray*}
    }
    where $r:=\min(p,q)$. Now we can use (\ref{eq:matrixnormdecay}) and proceed as above to
    show that the Hardy inequalities are satisfied for $d_i$ and $e_i$. Then, the application of
    Lemma \ref{lem:hardy} finishes the proof.
\end{proof}
As a corollary we can consider the shearlet frame $\Sigma$ constructed in \cite{Grohs2011a} and
briefly described in Subsection \ref{sec:shearletswithduals} and arrive at the following theorem.
\begin{theorem}\label{cor:normequiv}
    The curvelet frame $\Gamma^0$ and the shearlet frame $\Sigma$ constructed in \cite{Grohs2011a}
    span the same approximation spaces.
\end{theorem}
We remark that the same conclusion holds for the frames described in Subsections \ref{sec:smithframe}
and \ref{sec:borupframe}. Without proof we also mention that Theorem \ref{thm:normequiv} and Theorem
\ref{cor:normequiv} also hold for the spaces $G_{p,q}^\alpha$. The proof is similar but slightly
more technical.

We wish to stress that in fact this result for the first time proves the meta-theorem that curvelet- and shearlet
properties are equivalent.
\begin{remark}
    A similar result to Theorem \ref{cor:normequiv} has recently been shown in \cite{Labate2012}. The proofs in this
    paper only apply to bandlimited constructions which present considerably less technical difficulty.
\end{remark}
%

%
\section{Proof of Theorem \ref{thm:almostorth}}
\label{sec:proof}
%
The present sections presents the proof of our main result, namely the almost orthogonality of
any two systems of parabolic molecules of sufficient order. Since the argument is quite
involved we start by collecting some useful lemmata below in Subsection \ref{sec:estimates}
before we go on to the proof of the main result, Theorem \ref{thm:almostorth}
in Subsection \ref{sec:almostorth}.
%
\subsection{Some Estimates}\label{sec:estimates}
%
Here we collect several estimates which will turn out useful in the proof of
Theorem \ref{thm:almostorth}.
The following lemma can be found in \cite[Appendix K.1]{Grafakos2008}.
\begin{lemma}\label{lem:grafakos}
    For $N>1$ and $a,a'\in \mathbb{R}_+$, we have the inequality
    \begin{equation*}
        \int_\mathbb{R}\left(1+a|x|\right)^{-N}
        \left(1+a'|x-y|\right)^{-N}d\varphi\lesssim
        \max(a,a')^{-1}(1 + \min(a,a')|y|)^{-N}.
    \end{equation*}
\end{lemma}
As a corollary we can show the next result.
\begin{lemma}\label{lem:bumps}
    Assume that $|\theta| \le \frac{\pi}{2}$ and $N>1$. Then we have for $a,a'>0$
    the inequality
    \begin{equation}\label{eq:angulardecay}
        \int_{\mathbb{T}}\left(1+a|\sin(\varphi)|\right)^{-N}
        \left(1+a'|\sin(\varphi+\theta)|\right)^{-N}d\varphi\lesssim
        \max(a,a')^{-1}(1 + \min(a,a')|\theta|)^{-N}.
    \end{equation}
\end{lemma}
\begin{proof}
    For $\varphi\in
    \mathbb{T}$, we have the estimate
    $$
        |\sin(\varphi)|\geq \left\{\begin{array}{cc}
        |\varphi| & \varphi \in I_1:=\left [ -\frac{\pi}{2},\frac{\pi}{2}\right],\\
        |\varphi - \pi|&\varphi\in I_2:=\left [ \frac{\pi}{2},\pi\right],\\
        |\varphi + \pi|&\varphi\in I_3:=\left [-\pi,- \frac{\pi}{2}\right].
        \end{array}\right.
    $$
    In order to use Lemma \ref{lem:grafakos} we now split $\mathbb{T}$ into
    nine intervals depending on $\varphi + \theta, \varphi \in I_1,I_2,I_3$.
    Then the left-hand side of (\ref{eq:angulardecay}) can be estimated
    by nine terms of the form
    $$
        \int_\mathbb{\mathbb{R}}\left(1+a|\varphi|\right)^{-N}
        \left(1+a'|\varphi+\vartheta+\theta|\right)^{-N}d\varphi,
    $$
    where $\vartheta\in \left\{0,\pm\pi,\pm 2\pi \right\}$.
    By Lemma \ref{lem:grafakos} this expression can be bounded by a constant
    times
    $$
        \max(a,a')^{-1}(1 + \min(a,a')|\theta + \vartheta|)^{-N}.
    $$
    Now it remains to note that for $\vartheta\in \left\{\pm \pi,\pm 2\pi\right\}$
    and $|\theta|\le \frac{\pi}{2}$ we have $|\theta + \vartheta|\geq |\theta|$.
    This proves the lemma.
\end{proof}
Define the expression
$$
    S_{\lambda,M,N_1,N_2}(r,\varphi):=
    \min \left(1,2^{-s_\lambda}(1+r)\right)^{M}\left(1+2^{s_\lambda/2}|
        \sin(\varphi+\theta_\lambda)|\right)^{-N_2}\left(1 + 2^{-s_\lambda}r\right)^{-N_1}.
$$
The following lemma will be used in order to decouple the angular and the radial variables.
\begin{lemma}\label{lem:polarest}
    We have the estimate
    \begin{equation*}
         \min \left(1,2^{-s_\lambda }(1+r)\right)^{M}
        \left(1 + 2^{-s_\lambda}r
        \right)^{-N_1}
        \left(1 + 2^{-s_\lambda/2}r
        |\sin(\varphi + \theta_\lambda)|\right)^{-N_2} \lesssim
        S_{\lambda,M-L,N_1,L}(r,\varphi)
    \end{equation*}
    for every $0\le L \le N_2$.
\end{lemma}
\begin{proof}
    After picking $L$ we can estimate the quantity on the left hand side by
    $$
        \min \left(1,2^{-s_\lambda }(1+r)\right)^{M-L}\left(1 + 2^{-s_\lambda}r\right)^{- N_1}
        \left(\frac{\min \left(1,2^{-s_\lambda }(1+r)\right)}{1+2^{-s_\lambda/2}r|
        \sin(\varphi+\theta_\lambda)|}\right)^{L}.
    $$
    We need to show that
    \begin{equation}\label{eq:freqest_help}
        \frac{\min \left(1,2^{-s_\lambda }(1+r)\right)}{1+2^{-s_\lambda/2}r|
        \sin(\varphi+\theta_\lambda)|}\lesssim
        \left(1 + 2^{s_\lambda/2}|\sin(\varphi + \theta_\lambda)|\right)^{-1}.
    \end{equation}
    In order to prove (\ref{eq:freqest_help}), we distinguish three cases:
    \begin{itemize}
        \item {\bf $\mathbf{r\geq 2^{s_\lambda}}$: }  In this case we derive
        \begin{eqnarray*}
            \frac{\min \left(1,2^{-s_\lambda}(1+r)\right)}{1+2^{-s_\lambda/2}r|
            \sin(\varphi+\theta_\lambda)|}
            &\le &
            \frac{1}{1+2^{-s_\lambda/2}r|
            \sin(\varphi+\theta_\lambda)|}\le \frac{1}{1+2^{-s_\lambda/2}2^{s_\lambda}|
            \sin(\varphi+\theta_\lambda)|}
            \\
            &\le &
            \left(1 + 2^{s_\lambda/2}|\sin(\varphi + \theta_\lambda)|\right)^{-1}.
        \end{eqnarray*}
        \item {\bf $ \mathbf{r \le 1}$: }For $r\le 1$ we have
        \begin{eqnarray*}
            \frac{\min \left(1,2^{-s_\lambda}(1+r)\right)}{1+2^{-s_\lambda/2}r|
            \sin(\varphi+\theta_\lambda)|}&\lesssim & 2^{-s_\lambda}\lesssim
            \left(1 + 2^{s_\lambda/2}|\sin(\varphi + \theta_\lambda)|\right)^{-1}.
        \end{eqnarray*}
        \item {\bf $\mathbf{1<r<2^{s_\lambda}}$: }In this case we have
        \begin{eqnarray*}
            \frac{\min \left(1,2^{-s_\lambda}(1+r)\right)}{1+2^{-s_\lambda/2}r|
            \sin(\varphi+\theta_\lambda)|}
            &=&
            \frac{1+r}{r}\frac{1}{\frac{2^{s_\lambda}}{r} + 2^{s_\lambda/2}|\sin
            (\varphi + \theta)|}.
        \end{eqnarray*}
        Since $r>1$ we have that $\frac{1+r}{r}\le 2$ and since
        $r<2^{s_\lambda}$, we have that $\frac{2^{s_\lambda}}{r}\geq 1$. This
        proves the statement.
    \end{itemize}
\end{proof}
\begin{lemma}\label{lem:freqangdec}
    For $A,B>0$ and
    $$
        M > A-\frac{5}{4},\quad N_2\geq B,\quad N_1\geq A + 3/4,
    $$
    we have the estimate
    \begin{equation*}
        2^{-\frac{3}{4}(s_\lambda + s_{\mu})}\int_{\mathbb{R}_+}\int_{\mathbb{T}}
        S_{\lambda,M,N_1,N_2}(r,\varphi)S_{\mu,M,N_1,N_2}(r,\varphi)rdr d\varphi
        \lesssim 2^{-A\left| s_\lambda - s_{\mu}\right|}\left(
        1 + 2^{\min\left(s_\lambda,s_{\mu}\right)/2}|\theta_\lambda - \theta_{\mu}|\right)^{-B}.
    \end{equation*}
\end{lemma}
\begin{proof}
    We assume that $s_{\mu}\geq s_\lambda$ and start by showing
    the angular decay:
    By Lemma \ref{lem:bumps} and $N_2\geq B$, we have
    \begin{eqnarray*}
            2^{-\frac{3}{4}(s_\lambda + s_{\mu})}
            \int_{\mathbb{R}_+}\int_{\mathbb{T}}
            S_{\lambda,M,N_1,N_2}(r,\varphi)S_{\mu,M,N_1,N_2}(r,\varphi)rdr d\varphi
            &\lesssim &
            \mathcal{S}\cdot 2^{\frac{3}{4}(s_{\mu}-s_\lambda)}\left(1+2^{s_\lambda/2}
            |\theta_\lambda - \theta_{\mu}|\right)
            ^{-B},
    \end{eqnarray*}
    where
    \begin{equation}\label{eq:schoens}
        \mathcal{S}:=2^{-2s_{\mu} }\int_{\mathbb{R}_+}
        \min \left(1,2^{-s_\lambda}(1+r)\right)^{M}
        \min \left(1,2^{-s_{\mu}}(1+r)\right)^{M}
        \left(1 + 2^{-s_\lambda}r\right)^{-N_1}\left(1 + 2^{-s_{\mu}}r\right)^{-N_1}rdr.
    \end{equation}
    The remaining estimate $$\mathcal{S}\lesssim 2^{-(A+3/4)\left|s_\lambda - s_{\mu}\right|}$$
     is established by splitting up this integral
    into the four cases $r<1$, $1\le r < 2^{s_\lambda}$, $2^{s_\lambda}\le r < 2^{s_{\mu}}$
    and $r\geq 2^{s_{\mu}}$.
    \\
    {\it Case 1: $0\le r \le 1$}\\
    Here we only use the moment property and estimate
    \begin{eqnarray*}
        (\ref{eq:schoens})
        &\lesssim &
        2^{-2s_{\mu}}\int_0^1 2^{-s_\lambda M}2^{-s_{\mu}M} r^{2M+1}dr\\
        & \le &
        2^{-s_{\mu}(2+M)}2^{-s_\lambda M}\\
        &\le& 2^{-(A+3/4)(s_{\mu}-s_\lambda)}.
    \end{eqnarray*}
    \\
    {\it Case 2: $1\le r \le 2^{s_\lambda}$}\\
    For this case, we estimate
    \begin{eqnarray*}
        (\ref{eq:schoens})
        &\lesssim &
        2^{-2s_{\mu}}\int_1^{2^{s_\lambda}} 2^{-s_{\mu}M} r^{M}\left(2^{-s_\lambda}r\right)^{-N_1}r dr
        \\
        & = &2^{-(M+2)s_{\mu}}2^{N_1s_\lambda}\int_1^{2^{s_\lambda}}r^{M+1-N_1}dr
        \\
        &\le & 2^{-(M+2)s_{\mu}}2^{N_1s_\lambda}2^{(M+2-N_1)s_\lambda}\\
        & =&2^{-(M+2)(s_{\mu}-s_\lambda)}
        \\
        &\le & 2^{-(A+3/4)(s_{\mu}-s_\lambda)}.
    \end{eqnarray*}
    {\it Case 3: $2^{s_\lambda}\le r \le 2^{s_{\mu}}$}\\
    For this case, we estimate
    \begin{eqnarray*}
        (\ref{eq:schoens})
        &\lesssim &
        2^{-2s_{\mu}}\int_{2^{s_\lambda}}^{2^{s_{\mu}}}\left(2^{-s_{\mu}}r\right)^M
        \left(2^{-s_\lambda}r\right)^{-N_1}rdr
        \\
        &=&
        2^{-(2+M)s_{\mu}}2^{N_1s_{\lambda}}\int_{2^{s_\lambda}}^{2^{s_{\mu}}}r^{M+1-N_1}dr
        \\
        &\lesssim &
        2^{-(2+M)s_{\mu}}2^{N_1s_{\lambda}}2^{(M+2-N_1)s_{\mu}}\\
        &=&2^{-N_1(s_{\mu}-s_\lambda)}
        \\
        &\le & 2^{-(A+3/4)(s_{\mu}-s_\lambda)}.
    \end{eqnarray*}
    {\it Case 4: $2^{s_{\mu}}\le r$}\\
    For this case, we estimate
    \begin{eqnarray*}
        (\ref{eq:schoens})
        &\lesssim &
        2^{-2s_{\mu}}\int_{2^{s_{\mu}}}^\infty \left(2^{-s_\lambda}r\right)^{-N_1}
        \left(2^{-s_{\mu}}r\right)^{-N_1}rdr
        \\
        &=&
        2^{-2s_{\mu}}2^{N_1 s_{\mu}}2^{N_1s_{\lambda}}\int_{2^{s_{\mu}}}^\infty
        r^{-2N_1+1}dr
        \\
        &=&
        2^{-2s_{\mu}}2^{N_1 s_{\mu}}2^{N_1s_{\lambda}}2^{(-2N_1+2)s_{\mu}}
        \\
        &=&
        2^{-N_1(s_{\mu}-s_\lambda)}\\
        &\le& 2^{-(A+3/4)(s_{\mu}-s_\lambda)}.
    \end{eqnarray*}
    The proof is completed.
\end{proof}
%
%
\subsection{Almost Orthogonality}\label{sec:almostorth}
%
We now have all ingredients to prove our main result, which is Theorem \ref{thm:almostorth}.
\begin{proof}[Proof of Theorem \ref{thm:almostorth}]
    To keep the notation simple, we assume that $\theta_\lambda = 0$ and
    define $s_0:=\min(s_\lambda , s_{\mu})$.
    Further, we set
    $$
        \delta x := x_\lambda - x_\mu,\quad \delta \theta:= \theta_\lambda - \theta_\mu.
    $$
    By definition, we can write
    $$
        m_\lambda(\cdot) = 2^{\frac{3}{4}
         s_\lambda}a^{(\lambda)}\left(D_{2^{s_\lambda}}R_{\theta_\lambda}(\cdot - x_\lambda)\right),
        \quad
        p_\mu(\cdot) = 2^{\frac{3}{4} s_\mu}b^{(\mu)}\left(D_{2^{s_\mu}}R_{\theta_\mu}
        (\cdot - x_\mu)\right),
    $$
    where both $a^{(\lambda)}$ and $b^{(\mu)}$ satisfy (\ref{eq:molcond}).
    We have the equality
    \begin{eqnarray}\nonumber
            \left\langle m_\lambda,p_{\mu}\right\rangle
            &=&
            \left\langle \hat m_\lambda,\hat p_{\mu}\right\rangle
            =
            2^{-\frac{3}{4}(s_\lambda + s_{\mu})}\int_{\mathbb{R}^2}
            \hat a^{(\lambda)}\left(D_{2^{-s_\lambda}}R_{\theta_\lambda}\xi \right) \overline{\hat b^{(\mu)}
            \left(D_{2^{-s_\mu}}R_{\theta_\mu}\xi\right)}\exp\left(-2\pi i \xi \cdot \delta x\right)
            d\xi \nonumber \\
            \label{eq:parint1}
            &=&
            2^{-\frac{3}{4}(s_\lambda + s_{\mu})}\int_{\mathbb{R}^2}
            \mathcal{L}^k\left(\hat a^{(\lambda)}\left(D_{2^{-s_\lambda}}R_{\theta_\lambda}\xi \right) \overline{\hat b^{(\mu)}
            \left(D_{2^{-s_\mu}}R_{\theta_\mu}\xi\right)}\right)
            \mathcal{L}^{-k}\left(\exp\left(-2\pi i \xi \cdot \delta x\right)\right)
            d\xi,
    \end{eqnarray}
    where $\mathcal{L}$ is the symmetric differential operator (acting on the frequency
    variable) defined by
    $$
        \mathcal{L} :=
        I - 2^{s_0}\Delta_\xi - \frac{2^{2s_0}}{1 + 2^{s_0}|\delta \theta|^2}
        \frac{\partial^2}{\partial \xi_1^2}.
    $$
    We have
    \begin{equation}\label{eq:parint2}
        \mathcal{L}^{-k}\left(\exp\left(-2\pi i \xi \cdot \delta x\right)\right)
        = \left(1 + 2^{s_0}|\delta x|^2 + \frac{2^{2s_0}}{1 + 2^{s_0}|\delta \theta|}\langle
        e_\lambda , \delta x\rangle^2\right)^{-k}
        \exp\left(-2\pi i \xi \cdot \delta x\right),
    \end{equation}
    where $e_\lambda$ denotes the unit vector pointing in the direction described by the angle
    $\theta_\lambda$.
    By Lemma \ref{lem:diffopprod} and for $k\le \frac{R}{2}$, we have the inequality
    \begin{equation*}
        \mathcal{L}^k\left(\hat a^{(\lambda)}\left(D_{2^{-s_\lambda}}R_{\theta_\lambda}\xi \right)
         \overline{\hat b^{(\mu)}
        \left(D_{2^{-s_\mu}}R_{\theta_\mu}\xi\right)}\right)
        \lesssim S_{\lambda,M-N_2,N_1,N_2}(\xi)S_{\mu,M-N_2,N_1,N_2}(\xi).
    \end{equation*}
    Then, by (\ref{eq:parint1}) and (\ref{eq:parint2}) it follows that
    \begin{equation*}
        \left|\langle m_\lambda , p_\mu \rangle \right|
        \lesssim  2^{-\frac{3}{4}(s_\lambda + s_{\mu})}\hspace*{-0.2cm}\int_{\mathbb{R}^2}
        S_{\lambda,M-N_2,N_1,N_2}(\xi)S_{\mu,M-N_2,N_1,N_2}(\xi)d\xi
        \left(1 + 2^{s_0}|\delta x|^2 + \frac{2^{2s_0}}{1 + 2^{s_0}|\delta \theta|}\langle
        e_\lambda , \delta x\rangle^2\right)^{-k}
    \end{equation*}
    for all $k\le \frac{R}{2}$.
    Now we can use Lemma \ref{lem:freqangdec} and the fact that $R\geq 2N$ to establish that
    \begin{eqnarray*}
        \left|\langle m_\lambda , p_\mu \rangle \right|
        &\lesssim &
        2^{-2N|s_\lambda - s_{\mu}|}
        \left(1 + 2^{s_0}|\delta\theta|^2 \right)^{-N}
        \left(1 + 2^{s_0}|\delta x|^2 + \frac{2^{2s_0}}{1 + 2^{s_0}|\delta \theta|}\langle
        e_\lambda , \delta x\rangle^2\right)^{-N}
        \\
        &\le &
        2^{-2N|s_\lambda - s_{\mu}|}
        \left(1 + 2^{s_0}|\delta\theta|^2+
        2^{s_0}|\delta x|^2 + \frac{2^{2s_0}}{1 + 2^{s_0}|\delta \theta|}\langle
        e_\lambda , \delta x\rangle^2\right)^{-N}
        \\
        &\lesssim &
        2^{-2N|s_\lambda - s_{\mu}|}
        \left(1 + 2^{s_0}\left(|\delta \theta |^2 + |\delta x|^2+ |\langle e_\lambda ,
        \delta x\rangle|\right) \right)^{-N}
        = \omega(\lambda,\mu)^{-N}.
    \end{eqnarray*}
    The last inequality follows from the equation in the last line of the proof of \cite[Lemma 2.3]{CD02}.
    This proves the desired statement.
\end{proof}
\begin{lemma}\label{lem:diffopprod}
    Assume that (\ref{eq:conds}) holds for two systems of parabolic molecules of order
    $(R,M,N_1,N_2)$. Utilizing the notion of the proof of Theorem \ref{thm:almostorth},
    we have
    \begin{equation*}
        \mathcal{L}^k\left(\hat a^{(\lambda)}\left(D_{2^{-s_\lambda}}R_{\theta_\lambda}\xi \right)
         \overline{\hat b^{(\mu)}
        \left(D_{2^{-s_\mu}}R_{\theta_\mu}\xi\right)}\right)
        \lesssim S_{\lambda,M-N_2,N_1,N_2}(\xi)S_{\mu,M-N_2,N_1,N_2}(\xi)
    \end{equation*}
    for all $k\le R/2$.
\end{lemma}
\begin{proof}
    We show that
    \begin{multline}\label{eq:product}
        \left|\mathcal{L}^k\left(\hat a^{(\lambda)}\left(D_{2^{-s_\lambda}}R_{\theta_\lambda}
        \xi \right)
         \overline{\hat b^{(\mu)}
        \left(D_{2^{-s_\mu}}R_{\theta_\mu}\xi\right)}\right)\right|
        \lesssim
        \\
        \min \left(1,2^{-s_\lambda }(1+r)\right)^{M}
        \left(1 + 2^{-s_\lambda}r
        \right)^{-N_1}
        \left(1 + 2^{-s_\lambda/2}r
        |\sin(\varphi + \theta_\lambda)|\right)^{-N_2}
        \\
        \cdot \min \left(1,2^{-s_\mu }(1+r)\right)^{M}
        \left(1 + 2^{-s_\mu}r
        \right)^{-N_1}
        \left(1 + 2^{-s_\mu/2}r
        |\sin(\varphi + \theta_\mu)|\right)^{-N_2}
    \end{multline}
    which, using Lemma \ref{lem:polarest} with $L=N_2$, implies the desired statement.
    To show  (\ref{eq:product}) we use induction in $k$, namely we show that
    if we have two functions
    $a^{(\lambda)}, b^{(\mu)}$ satisfying
    (\ref{eq:molcond}) for $R,M,N_1,N_2$, then
    the expression
    $$
        \mathcal{L}\left(\hat a^{(\lambda)}\left(D_{2^{-s_\lambda}}R_{\theta_\lambda}\xi \right)
        \overline{\hat b^{(\mu)}
        \left(D_{2^{-s_\mu}}R_{\theta_\mu}\xi\right)}\right)
    $$
    can be written as a finite linear combination of terms of the form
    $$
        \hat c^{(\lambda)}\left(D_{2^{-s_\lambda}}R_{\theta_\lambda}\xi \right) \overline{\hat
        d^{(\mu)}
        \left(D_{2^{-s_\mu}}R_{\theta_\mu}\xi\right)}
    $$
    with $c,d$ satisfying (\ref{eq:molcond}) and $R$ replaced by $R-2$, see
    Lemma \ref{lem:prodrule}.
    Iterating this argument we can establish that for $k\le R/2$
    \begin{equation}\label{eq:estt}
        \mathcal{L}^k\left(\hat a^{(\lambda)}\left(D_{2^{-s_\lambda}}R_{\theta_\lambda}\xi \right)
        \overline{\hat b^{(\mu)}
        \left(D_{2^{-s_\mu}}R_{\theta_\mu}\xi\right)}\right)
    \end{equation}
    can be expressed as a finite linear combination of terms of the form
    \begin{equation}\label{eq:product1}
        \hat c^{(\lambda)}\left(D_{2^{-s_\lambda}}R_{\theta_\lambda}\xi \right) \overline{\hat
        d^{(\mu)}
        \left(D_{2^{-s_\mu}}R_{\theta_\mu}\xi\right)}
    \end{equation}
    with
    \begin{equation}\label{eq:product2}
        \left| \hat c^{(\lambda)}(\xi)\right|
        \lesssim \min\left(1,2^{-s_\lambda} + |\xi_1| + 2^{-s_\lambda/2}|\xi_2|\right)^M
        \left\langle |\xi|\right\rangle^{-N_1} \langle \xi_2 \rangle^{-N_2},
    \end{equation}
    and an analogous estimate for $d^{(\mu)}$.
    Combining (\ref{eq:product1}) and (\ref{eq:product2}), we obtain that
    \begin{multline*}
        |(\ref{eq:estt})| \lesssim
        \\
        \min \left(1,2^{-s_\lambda}+\left|\left(D_{2^{-s_\lambda}}R_{\theta_\lambda}\xi\right) _1
        \right|
        +2^{-s_\lambda /2}\left|\left(D_{2^{-s_\lambda}}R_{\theta_\lambda}\xi \right)_2 \right|
        \right)^M
        \left\langle \left |D_{2^{-s_\lambda}}R_{\theta_\lambda}\xi\right|\right\rangle^{-N_1}
        \langle \left|\left(D_{2^{-s_\lambda}}R_{\theta_\lambda}\xi \right)_2 \right|  \rangle^{-
        N_2}
        \\
        \cdot \min \left(1,2^{-s_\mu}+\left|\left(D_{2^{-s_\mu}}R_{\theta_\mu}\xi\right) _1 \right|
        +2^{-s_\mu /2}\left|\left(D_{2^{-s_\mu}}R_{\theta_\mu}\xi \right)_2 \right|\right)^M
        \left\langle \left |D_{2^{-s_\mu}}R_{\theta_\mu}\xi\right |\right\rangle^{-N_1}
        \langle \left|\left(D_{2^{-s_\mu}}R_{\theta_\mu}\xi \right)_2 \right|  \rangle^{-N_2}.
    \end{multline*}
    Transforming this inequality into polar coordinates as in (\ref{eq:moldecaypolar}) yields (\ref{eq:product}).
    This finishes the proof.
\end{proof}
\begin{lemma}\label{lem:prodrule}
    Given two functions
    $a^{(\lambda)}, b^{(\mu)}$ satisfying
    (\ref{eq:molcond}) for $R,M,N_1,N_2$. Then
    the expression
    $$
        \mathcal{L}\left(\hat a^{(\lambda)}\left(D_{2^{-s_\lambda}}R_{\theta_\lambda}\xi \right)
        \overline{\hat b^{(\mu)}
        \left(D_{2^{-s_\mu}}R_{\theta_\mu}\xi\right)}\right)
    $$
    can be written as a finite linear combination of terms of the form
    $$
        \hat c^{(\lambda)}\left(D_{2^{-s_\lambda}}R_{\theta_\lambda}\xi \right) \overline{\hat
        d^{(\mu)}
        \left(D_{2^{-s_\mu}}R_{\theta_\mu}\xi\right)}
    $$
    with $c,d$ satisfying (\ref{eq:molcond}) for $R-2,M,N_1,N_2$.
\end{lemma}
\begin{proof}
    Recall the definition
    $$
        \mathcal{L} :=
        I - 2^{s_0}\Delta_\xi - \frac{2^{2s_0}}{1 + 2^{s_0}|\delta \theta|^2}
        \frac{\partial^2}{\partial \xi_1^2}
    $$
    To show this statement we treat the three summands of the operator $\mathcal{L}$
    separately. The first part is the identity, and therefore the statement is trivial.
    To handle the second part, the frequency Laplacian
    $2^{s_0}\Delta$, we use the product rule
    \begin{equation*}
        \Delta(fg) = 2\left(\partial^{(1,0)}f\partial^{(1,0)}g
        + \partial^{(0,1)}f\partial^{(0,1)}g\right)
        + (\Delta f)g + f(\Delta g).
    \end{equation*}
    Therefore we need to estimate the derivatives of degree $1$
    and the Laplacians of the two factors in the product
    \begin{equation*}
        \hat a^{(\lambda)}\left(D_{2^{-s_\lambda}}R_{\theta_\lambda}\xi \right)
        \overline{\hat b^{(\mu)}
        \left(D_{2^{-s_\mu}}R_{\theta_\mu}\xi\right)}=:
        A(\xi)B(\xi).
    \end{equation*}
    We start with the first factor,
    $$
        A(\xi)=\hat a^{(\lambda)}\left(2^{-s_\lambda}\cos(\theta_\lambda)\xi_1
        - 2^{-s_\lambda}\sin(\theta_\lambda)\xi_2 ,
        2^{-s_\lambda/2}\sin(\theta_\lambda)\xi_1
        + 2^{-s_\lambda/2}\cos(\theta_\lambda)\xi_2\right).
    $$
    Define
    \begin{equation*}
        A_1(\xi):=\partial^{(1,0)}\hat a^{(\lambda)}\left(D_{2^{-
        s_\lambda}}R_{\theta_\lambda}\xi\right)
        \ \mbox{ and }\
        A_2(\xi):=\partial^{(0,1)}\hat a^{(\lambda)}\left(D_{2^{-
        s_\lambda}}R_{\theta_\lambda}\xi\right).
    \end{equation*}
    By definition, the functions $A_1,\ A_2$ satisfy (\ref{eq:molcond}) with
    $R$ replaced by $R-1$.
    An application of the chain rule shows that
    \begin{equation*}
        \partial^{(1,0)}A(\xi)
        =2^{-s_\lambda}\cos(\theta_\lambda)A_1(\xi)
        + 2^{-s_\lambda/2}\sin(\theta_\lambda)A_2(\xi).
    \end{equation*}
    Analogously, one can compute
    \begin{equation*}
        \partial^{(0,1)}A(\xi)
        =-2^{-s_\lambda}\sin(\theta_\lambda)A_1(\xi)
        + 2^{-s_\lambda/2}\cos(\theta_\lambda)A_2(\xi),
    \end{equation*}
    and the exact same expressions for $B$ using the obvious definitions
    for $B_1,\ B_2$.
    We get
    \begin{multline*}
        \partial^{(1,0)}A\partial^{(1,0)}B
        =2^{-s_\lambda-s_\mu}\cos(\theta_\lambda)
        \cos(\theta_\mu)A_1B_1
        +2^{-s_\lambda/2-s_\mu}\sin(\theta_\lambda)\cos(\theta_\mu)
        A_2B_1\\
        +2^{-s_\mu/2-s_\lambda}\sin(\theta_\mu)\cos(\theta_\lambda)
        A_1B_2
        + 2^{-s_\lambda/2-s_\mu/2}\sin(\theta_\lambda)
        \sin(\theta_\mu)A_2B_2.
    \end{multline*}
    It follows that
    $2^{s_0}\partial^{(1,0)}A\partial^{(1,0)}B$ can be written
    as a linear combination as claimed (recall that $s_0 = \min(s_\lambda,s_\mu))$.
    The same argument applies to
    the product $2^{s_0}\partial^{(0,1)}A\partial^{(0,1)}B$.

    It remains to consider the factor
    $$
        (\Delta A)B + A(\Delta B),
    $$
    where, for symmetry reasons, we only treat the summand
    $$
        (\Delta A)B.
    $$
    In fact, it suffices to only consider
    \begin{equation*}
        (\partial^{(2,0)}A)B
        =
        \left(2^{-2s_\lambda}\cos(\theta_\lambda)^2A_{11}
        +2^{-3s_\lambda/2+1}\sin(\theta_\lambda)
        \cos(\theta_\lambda)A_{12}
        - 2^{-s_\lambda}\sin(\theta_\lambda)^2A_{22}\right)B
    \end{equation*}
    with $A_{ij}$ defined in an obvious way, satisfying (\ref{eq:molcond})
    with $R$ replaced by $R-2$. The term $(\partial^{(2,0)}A)B$, and hence
    $(\Delta A)B$, can be handled in the same way, as can $A(\Delta B)$.
    This takes care of the term $2^{s_0}\Delta$ in the definition of
    $\mathcal{L}$.

    Finally we need to handle the last term in the definition of $\mathcal{L}$,
    namely
    $$
        \frac{2^{2s_0}}{1 + 2^{s_0}| \theta_\mu|^2}
        \frac{\partial^2}{\partial \xi_1^2}
    $$
    for $\theta_\lambda = 0$ (otherwise the second order derivative
    would be in the direction of the unit vector with angle $\theta_\lambda$ with
    obvious modifications in the proof).
    With our notation and using the product rule we need to consider terms of the form
    $$
        \left(\partial^{(2,0)}A\right) B,\quad
        \left(\partial^{(1,0)}A\right)\left(\partial^{(1,0)} B\right),\quad
        A\left(\partial^{(2,0)} B\right),
    $$
    and show that each of them, multiplied by the factor
    $\frac{2^{2s_0}}{1 + 2^{s_0}|\theta_\mu|^2}$, satisfies the desired representation.
    Let us start with
    $$
        \left(\partial^{(2,0)}A\right) B,
    $$
    which, using the fact that $\sin(\theta_\lambda) = 0$, can be written as
    $$
        2^{-2s_\lambda}A_{11}B,
    $$
    and which clearly satisfies the desired assertion.
    Now consider the expression
    $$
        \left(\partial^{(1,0)}A\right)\left(\partial^{(1,0)} B\right)
    $$
    which can be written as
    $$
        2^{-s_\lambda}2^{-s_\mu}\cos(\theta_\mu)A_1B_1
        + 2^{-s_\lambda}2^{-s_\mu/2}\sin(\theta_\mu)A_1B_2.
    $$
    The first summand in this expression clearly causes no problems.
    To handle the second term we need to show that
    \begin{equation}\label{eq:diffopbound}
        \frac{2^{2s_0}}{1 + 2^{s_0}|\theta_\mu|^2}
        2^{-s_\lambda}2^{-s_\mu/2}\sin(\theta_\mu) \lesssim 1.
    \end{equation}
    Here we have to distinguish two cases. First, assume that $|\theta_\mu|\le 2^{-s_0/2}$.
    Then we can estimate
    $$
        \sin(\theta_\mu)\lesssim 2^{-s_0/2},
    $$
    which readily yields the desired bound for (\ref{eq:diffopbound}).
    For the case $|\theta_\mu|\geq 2^{-s_0/2}$ we estimate
    $$
        \frac{2^{2s_0}}{1 + 2^{s_0}|\theta_\mu|^2}
        2^{-s_\lambda}2^{-s_\mu/2}\sin(\theta_\mu)
        \lesssim
        \frac{2^{2s_0}}{1 + 2^{s_0/2}|\theta_\mu|}
        2^{-s_0}2^{-s_0/2}|\theta_\mu|
        \le
        \frac{2^{2s_0}}{2^{s_0/2}|\theta_\mu|}
        2^{-s_0}2^{-s_0/2}|\theta_\mu| = 1
    $$
    which shows (\ref{eq:diffopbound}) also for this case.

    We are left with estimating the term
    $$
        A\left(\partial^{(2,0)} B\right)
    $$
    which can be written as
    $$
        2^{-2s_\mu}\cos(\theta_\mu)^2 A B_{11}
        + 2^{-3s_\mu/2+1}\sin(\theta_\mu)\cos(\theta_\mu)AB_{12}
        + 2^{-s_\mu}\sin(\theta_\mu)^2A B_{22}.
    $$
    The first two terms are of a form already treated, and the last term can
    be handled using the fact that $\sin(\theta_\mu)^2\le \theta_\mu^2$.
\end{proof}
%

%
\section*{Acknowledgements}
%
G.~Kutyniok would like to thank Wolfgang Dahmen, David Donoho, Wang-Q
Lim, and Pencho Petrushev for
enlightening discussions on this and related topics.
She acknowledges support by the Einstein Foundation Berlin, by
Deutsche Forschungsgemeinschaft
(DFG) Grant SPP-1324 KU 1446/13 and DFG Grant KU 1446/14, and by the
DFG Research Center {\sc Matheon}
``Mathematics for key technologies'' in Berlin.
The research of P.~Grohs was in part funded by the
European Research Council under Grant ERC Project STAHDPDE No. 247277.
\bibliographystyle{abbrv}
\bibliography{molecules}
\end{document}